\documentclass[a4paper]{article}
\usepackage[normalem]{ulem}
\usepackage{graphicx} % Required for inserting images
\usepackage{amsmath,amssymb,amsthm}
\usepackage{enumitem}
\usepackage[margin=1.0in]{geometry}
\usepackage{cancel}
\usepackage{comment}

\newcommand{\sumjni}{\sum_{\substack{j=1\\j \neq i}}^N}

\allowdisplaybreaks

\newcommand{\R}{\mathbb{R}}

\usepackage{xcolor}

\newtheorem{theorem}{Theorem}[section]
\newtheorem{lemma}[theorem]{Lemma}
\newtheorem{corollary}[theorem]{Corollary}
\newtheorem{proposition}[theorem]{Proposition}

\usepackage[backend=bibtex,maxbibnames=99]{biblatex}
\usepackage{cleveref}

\newcommand{\inner}[2]{\langle #1,#2\rangle}

\newcommand{\Rl}{\mathcal I_R}

\numberwithin{equation}{section}

\title{Symmetric Polynomials Method for the Multiple Loewner Chains}
\author{Sergio Andraus, Nicole Hufnagel, Jacek Małecki, Vlad Margarint}
\date{\today}

\begin{document}

\maketitle
\begin{abstract}
We introduce the technique of Symmetric Polynomials in the study of Loewner chains driven by a family of stochastic particle systems including but not limited to Dyson Brownian motion. As a first application of this technique we obtain the phase transitions at the critical parameter $\kappa=4$ in terms of the behavior of a class of symmetric polynomials evolving in time. We obtain new proofs of the phase transitions, we also obtain new closed-form expressions relevant for the models. In addition, we apply this method for some initial deterministic multiple Loewner evolutions. Finally, as this is, to our knowledge, the first paper analyzing this method in the context of multiple Loewner chains driven by root-system based multivariate processes, we also give an overview of potential future applications of this technique in the study of these models. 

\end{abstract}

\section{Introduction}

The method of Symmetric polynomials is a powerful technique which allows us to study the union of hulls generated by the Loewner equation as a `unique notion' of boundary. Starting from a general Loewner theory/Complex Dynamics angle, we first consider some general models in which we study classes interacting particle systems as drivers of the Multiple Loewner Equations. The introduction of symmetric polynomials that relate the Loewner dynamics and the drivers allows us via `symmetrization' of these elements to study new dynamics which allow us to detect phase transitions in the parameters. In addition, we study related objects in the case of Multiple Loewner Equation driven by Dyson Brownian Motion, which has connections with Mathematical Physics and Conformal Field Theory, see \cite{Cardy}. In that case, in addition to new proofs of the phase transitions, we also obtain new closed-form expressions relevant for the models. Moreover, we perform some initial study on some deterministic multiple Loewner dynamics models. All of the aforementioned developments can be further generalized and investigated.

Schramm-Loewner evolution (SLE), introduced by Schramm \cite{Schramm2000}, describes conformally invariant random growth and arises in scaling limits of planar statistical mechanics models; see, for example, \cite{k2,k3,k6}. In the chordal setting, the Loewner equation in the upper half-plane is
    \begin{equation*}
        \partial_t g_t(z) = \frac{2}{g_t(z)-W(t)}, \qquad g_0(z)=z.
    \end{equation*}
For a continuous real driving function $W$, its solutions define conformal maps $g_t:\mathbb H\setminus K_t\to\mathbb H$ associated with growing hulls $K_t$. The parametrization is fixed by
    \begin{equation*}
        g_t(z)=z+\frac{2t}{z}+O(z^{-2}) \qquad\text{as }z\to\infty.
    \end{equation*}
Chordal $\mathrm{SLE}_\kappa$ is obtained by taking $W(t)=\sqrt{\kappa}B(t)$, where $B$ is a standard one-dimensional Brownian motion and $\kappa\geq0$; see \cite{BerestyckiNorris}.

\medskip

One generalization of the SLE model that has received attention in recent years is a multi-driver version of the Loewner differential equation. One such case is the multiple SLE model, where the single Brownian driving function is replaced by a system of particles following Dyson Brownian motion (DBM). These particles evolve under independent Brownian noise and interact through inverse-distance repulsion. For works in this direction, see \cite{Toma,KaVlad,JiamVlad,HealeyLawler,abuzaid2024large}. Other formulations of multiple SLE are studied in \cite{Peltola,Dub,Zhan}. These models also have connections with conformal field theory; see, for example, \cite{Cardy}. We use the normalization of Dyson Brownian motion corresponding to $\beta=8/\kappa$, $\kappa>0$, in which the particles satisfy \cite{Hotta_2018}
\begin{equation}
\label{eq:sde:dyson}
    d \lambda_i (t) = \sqrt{\kappa}\,dB_i(t) +  \sumjni \frac{4\,}{\lambda_i(t)-\lambda_j(t)}dt, \quad  i = 1, \dots, N.
\end{equation}
For simplicity, we denote the driver configuration by $\lambda(t):=(\lambda_1(t),\ldots,\lambda_N(t))$.

The corresponding multiple Loewner equation, with weight $2$ assigned to each driver, is
\begin{equation}
\label{eq:multipleSLE}
    \partial_{t} g_{t}(z) = \sum_{i=1}^{N}\frac{2}{g_{t}(z)-\lambda_i(t)}, \quad g_0(z) = z.
\end{equation}

There are instances where the rhs of this equation has a denominator of $N$; we consider the version of the multiple Loewner equation given here to keep the format consistent with the one in \cite{Katori_2021}. We study the evolution of real initial points under these maps. For an initial point $x$ distinct from the drivers, the solution $g_t(x)$ remains real up to its first encounter with a driver or the first collision among the driving particles. Our questions concern whether a driver is reached, which driver is reached, and how this hitting time compares with particle collisions.

\medskip

In the one-driver case, the gap $X_t=g_t(x)-W(t)$ satisfies
    \begin{equation*}
        dX_t=\frac{2}{X_t}\,dt-\sqrt{\kappa}\,dB_t.
    \end{equation*}
For $\kappa>0$, the process $|X_t|/\sqrt{\kappa}$ is a Bessel process of dimension $1+4/\kappa$ up to its first zero. Its boundary classification gives a transition at $\kappa=4$, see \cite{BerestyckiNorris}. This suggests seeking scalar observables that play a similar role for several interacting drivers. Symmetric polynomials provide such observables. The sum in \eqref{eq:multipleSLE} is unchanged by relabelling the drivers, and symmetric functions of the gaps $g_t(x)-\lambda_i(t)$ have the same property. We consider elementary symmetric polynomials in the squared gaps. These are nonnegative, and their stochastic dynamics exhibit cancellations of the singular particle-interaction terms. Related polynomial methods have been used in the analysis of singular interacting particle systems. In \cite{GraczykMalecki}, they yield strong existence and pathwise uniqueness under suitable assumptions. In \cite{AHM}, symmetric polynomials in squared root projections are used to study collision types and the Hausdorff dimension of collision times. Here, we apply these methods to the interaction between a Loewner image and its driving particles.

\medskip

We study the applications of this method in various contexts. One of them is linked with deterministic Loewner theory. There is a significant interest of the community in the deterministic direction, see, for example, \cite{LindRohde,LindRohde2,JinwooYilin}, in the one-curve case (see also \cite{stochasticversionlindrohdepaper} for a probabilistic study in the context of SLE).

Small-noise limits as $\kappa\to0$ have also been studied; see, for example, \cite{abuzaid2024large,PeltolaWangdeterministic}  for a non-exhaustive list of references.  In addition, there is an extended interest in the study of the multiple SLE with Dyson Brownian Motion drivers due to its connections with Physics, cf. \cite{Hotta_2018,hotta2021limits}. See also the non-exhaustive list of references on Multiple SLE analysis \cite{Peltola,berestycki2025uniform,Karrila2020,kim2026numerical,LennelsViklund,peltola2019global,Zhan}.

One quantity that appears in our analysis is the hitting time of the driver under the Loewner dynamics started from some initial position. This quantity is linked in the one-curve case with the conformal welding, which is a very important object of study (see \cite{MargarintTim} for a deterministic case study, see \cite{ang2025liouville,ang2023cutting,ang2026conformal,astala2011random,Vladweld,duplantier2021liouville,Kavvadias2026,McEnteggart2021,sheffield2016conformal,sun2024sle}, for related studies in the random case).

\medskip

We extend the multiple SLE driven by Dyson Brownian motion model to multiple Loewner chains in two directions. The first direction concerns the driving particle system. We allow the diffusion coefficients, regular drift, and singular repulsion coefficients to depend on time and on the entire particle configuration. A general framework for such systems was introduced by Małecki \cite{Malecki2026}. It encompasses Dyson-type systems, radial Dunkl processes, squared-Bessel and Wishart particle systems, and various non-colliding diffusion models. The repulsion coefficients may also vanish on the corresponding chamber walls. It is known that these interacting particle systems have unique strong solutions. We use these solutions as drivers of the multiple Loewner chains. We consider the dynamics on the real line of these models up to the first singularity that occurs.

\medskip

The second direction concerns the Loewner evolution itself. We associate the driving particle system with a finite family of interface directions invariant under its reflection group. This construction gives a root-adapted Loewner equation and a corresponding family of symmetric polynomial observables. We derive their stochastic dynamics for a general finite reduced root system and establish boundary-hitting results for the classical systems of types $A_{N-1}$, $B_N$, and $D_N$. For the sake of simplicity, we omit the range of these root systems when they are in a sub/superscript.

\medskip

The symmetric polynomials method is one way to extend the Bessel analysis on the real line of the one-curve SLE case. In the one-curve SLE case, one can relate the behavior of the Bessel process on the real line with the SLE traces to obtain phase transitions \cite{BerestyckiNorris}. The new method of symmetric polynomials aims to provide an extension of this analysis to multiple Loewner chains. In the general multiple Loewner chain models we consider, the geometry of the hulls is less understood, but we envision that the phase transitions and other analyses that we perform in this work will provide insights on these in the future.

\medskip

Let $e_1,\ldots,e_N$ be the standard basis in $\R^N$, and let $R\subset\R^N$ be a finite reduced root system. We fix a choice of positive roots $R_+$ and denote by $W_R$ the corresponding reflection group. 
The associated open Weyl chamber is
\begin{equation*}
    C_R := \left\{ x\in\R^N: \langle\alpha,x\rangle>0 \text{ for every }\alpha\in R_+ \right\}.
\end{equation*}
Its closure and boundary are denoted by $\overline{C_R}$ and $\partial C_R$, respectively. We consider the generalized Bessel-Dunkl particle system
\begin{equation}
    \label{eq:general-root-particle-system}
    d\lambda_i(t) = \sigma_i\bigl(t,\lambda(t)\bigr)\,dB_i(t) + b_i\bigl(t,\lambda(t)\bigr)\,dt  + \sum_{\alpha\in R_+} k_\alpha\bigl(t,\lambda(t)\bigr)\frac{\alpha_i}{\langle\alpha,\lambda(t)\rangle}\,dt, \qquad i=1,\ldots,N.
\end{equation}
Here, $B(t) = \bigl(B_1(t),\ldots,B_N(t) \bigr)$ is an $N$-dimensional Brownian motion with independent components. The functions $\sigma_i,b_i: [0,\infty)\times\overline{C_R} \longrightarrow\R$ are the diffusion and regular drift coefficients, while $k_\alpha: [0,\infty)\times\overline{C_R} \longrightarrow[0,\infty)$, $\alpha\in R_+$, determine the strength of the repulsion from the root hyperplanes. Thus, the repulsive force is no longer determined by constant root multiplicities. Its strength may vary with time and with the particle configuration, although its direction remains normal to the corresponding root hyperplane. The classical radial Dunkl process is recovered by taking $\sigma_i\equiv1$ and $b_i\equiv0$ for every $i$, and $k_\alpha(t,x)\equiv k(\alpha)$, where $k$ is a $W_R$-invariant multiplicity function.

For the SDE calculations below, the system is initially considered in the interior of the Weyl chamber, up to the
 stopping time
\begin{equation}
    \label{eq:general-chamber-boundary-time}
    \tau_{\rm coll} := \inf\left\{t>0: \langle\alpha,\lambda(t)\rangle=0\text{ for some }\alpha\in R_+\right\},
\end{equation}
which we will refer to as \emph{collision time} and involves only the drivers.
All calculations are performed before $\tau_{\rm coll}$. We assume that the particle equation admits a continuous non-explosive semimartingale solution on this interval and that its coefficients satisfy the local integrability conditions required for the stochastic integrals below.

Let $e_0$ be a unit vector orthogonal to $\R^N$, and identify $\R^{N+1}=\R e_0\oplus\R^N$. We regard every vector $v\in\R^N$ as the vector $(0,v)\in\R^{N+1}$. Throughout, $\langle\cdot,\cdot\rangle$ denotes the Euclidean inner product in the ambient space. Thus, for $u,v\in\R^{N+1}$,
    \begin{equation*}
        \langle u,v\rangle=\sum_{i=0}^N u_i v_i.
    \end{equation*}
Its restriction to the embedded copy of $\R^N$ is the usual Euclidean inner product on $\R^N$. Extend the action of $W_R$ to $\R^{N+1}$ by fixing $e_0$. Given a finite $W_R$-invariant set $\mathcal V_R\subset\R^N$, define
    \begin{equation}
        \label{eq:general-interface-set}
        \Rl:=\{e_0-v:v\in\mathcal V_R\}.
    \end{equation}

For $z\in\mathbb H$, and for real initial points at which the Loewner equation below is nonsingular, put
    \begin{equation}
        \label{eq:extended-loewner-state}
        \eta(t,z):=g_t(z)e_0+\sum_{i=1}^N\lambda_i(t)e_i,
    \end{equation}
and define the interface coordinates by
    \begin{equation}
        \label{eq:interface-coordinate}
        X_\xi(t,z):=\langle\xi,\eta(t,z)\rangle, \qquad \xi\in\Rl.
        % X_\xi(t,z):=\langle\xi,\eta(t,z)\rangle_{0}, \qquad \xi\in\Rl.
    \end{equation}
The real linear forms are extended complex linearly. If $\xi=e_0-v$, then
    \begin{equation*}
        X_\xi(t,z)=g_t(z)-\langle v,\lambda(t)\rangle.
    \end{equation*}
Let $a_\xi(t)$, $\xi\in\Rl$, be nonnegative, continuous, adapted, locally bounded processes. For each $\xi$, assume either $a_\xi\equiv0$ or $a_\xi(t)>0$ throughout the interval under consideration. The associated Loewner equation is
    \begin{equation}
        \label{eq:general-root-adapted-loewner}
        \partial_tg_t(z) =\sum_{\xi\in\Rl}\frac{a_\xi(t)}{X_\xi(t,z)}, \qquad g_0(z)=z.
    \end{equation}
A summand with identically zero coefficient is understood to be zero, including at a zero of its denominator.

For an initial point distinct from the poles at time zero, let
    \begin{equation*}
        %\label{eq:general-interface-hitting-time}
        \tau_z :=\inf\left\{t>0:\liminf_{s\uparrow t}|X_\xi(s,z)|=0\text{ for some }\xi\in\Rl\text{ with }a_\xi(t)>0\right\},
    \end{equation*}
where $\inf\varnothing=\infty$. The Loewner evolution is considered for $t<\tau_z\wedge\tau_{\rm coll}$. For the classical root systems, we take 

\begin{equation}
\label{eq:interface:directions}
\begin{aligned}
    \mathcal V_A
        &=\{e_i:1\leq i\leq N\},
    &\mathcal I_A
        &=\{e_0-e_i:1\leq i\leq N\},\\
    \mathcal V_B
        &=\{0\}\cup\{\pm e_i:1\leq i\leq N\},
    &\mathcal I_B
        &=\{e_0\}\cup\{e_0-e_i,e_0+e_i:1\leq i\leq N\},\\
    \mathcal V_D
        &=\{\pm e_i:1\leq i\leq N\},
    &\mathcal I_D
        &=\{e_0-e_i,e_0+e_i:1\leq i\leq N\}.
\end{aligned}
\end{equation}
These interface vectors are precisely the roots with $e_0$ coordinate equal to $1$ in the corresponding extensions $A_N$, $B_{N+1}$, and $D_{N+1}$. The interface coordinates are $g_t(z)-\lambda_i(t)$ in type $A_{N-1}$, $g_t(z)$ and $g_t(z)\pm\lambda_i(t)$ in type $B_N$, and $g_t(z)\pm\lambda_i(t)$ in type $D_N$. The coefficients of the classical models are specified in Section~\ref{sec:classical-phase-transitions}.

\medskip

For real initial points, the coordinates $X_\xi(t,x)$ are signed gaps between the Loewner image and the interface points. The action of $W_R$ permutes these coordinates, so symmetric polynomials in their squares are invariant functions of the particle configuration. This algebraic symmetry underlies the cancellation of the particle-root denominators in their stochastic dynamics.

\medskip

In Section~\ref{sec:SymmPolyn}, we derive these dynamics for the general particle system and interface family. In Section~\ref{sec:classical-phase-transitions}, we specialize to the classical models. When the highest polynomial has a strictly positive initial value, a multiplicative normalization and a time change yield a squared Bessel process of dimension $1+4/\kappa$.

\medskip

For the classical coefficients specified there, the particle chamber boundary is inaccessible when $0<\kappa\leq8$ and is reached almost surely when $\kappa>8$. For $0<\kappa\leq8$, each fixed admissible real initial point almost surely never meets a driver when $\kappa\leq4$ and hits a driver in finite time almost surely when $\kappa>4$. This conclusion holds in types $A_{N-1}$ and $B_N$ and in the noncentral components of type $D_N$. In the central component of type $D_N$, the first hitting time of a driver is almost surely finite throughout $0<\kappa\leq8$.

\medskip

Section~\ref{sec:collisions} examines two type-$A_1$ drivers and an initial point between them. A time-changed gap-ratio diffusion gives the probabilities of hitting either driver and determines the ordering of hitting and particle collision. We subsequently consider deterministic examples and discuss further questions in Section~\ref{sec:outlook}.

\section{Root systems and notation}
\label{sec:root-systems}

For a nonzero vector $\beta\in\R^N$, let
    \begin{equation*}
        s_\beta(y) :=y-\frac{2\langle y,\beta\rangle}{|\beta|^2}\beta, \qquad y\in\R^N,
    \end{equation*}
be the reflection in the hyperplane orthogonal to $\beta$. A finite reduced root system is a finite set
$R\subset\R^N\setminus\{0\}$ such that
    \begin{equation*}
        s_\beta(R)=R, \qquad R\cap\R\beta=\{\beta,-\beta\}, \qquad \beta\in R.
    \end{equation*}
The roots need not span the whole ambient space. The group generated by the reflections $s_\beta$, $\beta\in R$, is denoted by $W_R$.

Choose $h\in\R^N$ such that
$\langle\beta,h\rangle\neq0$ for every $\beta\in R$.
The corresponding positive roots are
    \begin{equation*}
        R_+:=\{\beta\in R:\langle\beta,h\rangle>0\},
    \end{equation*}
so that $R=R_+\cup(-R_+)$.

For the classical systems considered below, we use
    \begin{align*}
        (A_{N-1})_+ &=\{e_j-e_i:1\leq i<j\leq N\},\\
        (D_N)_+     &=\{e_j-e_i,e_j+e_i:1\leq i<j\leq N\},\\
        (B_N)_+     &=(D_N)_+\cup\{e_i:1\leq i\leq N\}.
    \end{align*}
Thus $|e_j-e_i|^2=|e_j+e_i|^2=2$ and $|e_i|^2=1$. The group $W_A$ acts by coordinate permutations, $W_B$ by signed coordinate permutations, and $W_D$ by signed coordinate permutations with an even number of sign changes. We use the embedding $\R^N\subset\R^{N+1}$ introduced above. Particle roots have zero $e_0$ coordinate, whereas every $\alpha\in\mathcal I_R$ has $\alpha_0=1$. The action of $W_R$ on $\R^{N+1}$ fixes $e_0$.

\section{Symmetric polynomials with the generalized Loewner dynamics}
\label{sec:SymmPolyn}

We consider symmetric polynomials associated with the particle system \eqref{eq:general-root-particle-system} and the Loewner evolution \eqref{eq:general-root-adapted-loewner}. Throughout this section, $x\in\R$ and $t<\tau_x\wedge\tau_{\rm coll}$. Whenever we mention a hitting time, it will refer to $\tau_x$, that is, $g_t(x)$ hitting one of the drivers. The interface coordinates $X_\xi(t,x)$ are real, and the zero-coefficient convention from the introduction applies to the Loewner equation. Put $M:=|\Rl|$. For $I\subseteq\Rl$ and an integer $n$, define
    \begin{equation}
        \label{eq:deleted-symmetric-polynomial}
        E_n^I(t,x) := \sum_{\substack{J\subseteq\Rl\setminus I\\|J|=n}} \prod_{\xi\in J}X_\xi(t,x)^2.
    \end{equation}
We use the conventions
    \begin{equation*}
         E_0^I=1, \qquad E_n^I=0 \quad\text{if }n<0 \text{ or }n>|\Rl\setminus I|.
    \end{equation*}
Write $E_n:=E_n^\varnothing$. In particular,
    \begin{equation*}
        %\label{eq:highest-full-interface-polynomial}
        E_M(t,x)=\prod_{\xi\in\Rl}X_\xi(t,x)^2.
    \end{equation*}
These polynomials are nonnegative. The polynomial $E_M$ vanishes exactly when at least one interface coordinate vanishes; hitting corresponds to vanishing of a coordinate with positive Loewner weight. Since $\mathcal V_R$ is $W_R$-invariant, the action of $W_R$ on $\R^{N+1}$ fixing $e_0$ permutes $\Rl$. For fixed $g$, the transformation $\lambda\mapsto w\lambda$ therefore permutes the squared interface coordinates. Consequently, each $E_n$, regarded as a polynomial in $(g,\lambda)$, is invariant under this action. The generating polynomial is
    \begin{equation}
        \label{eq:squared-gap-generating-polynomial}
        \mathcal P_t(u,x) := \prod_{\xi\in\Rl} \bigl(u+X_\xi(t,x)^2\bigr) = \sum_{n=0}^M E_n(t,x)u^{M-n}.
    \end{equation}
    The arguments $(t,x)$ are suppressed whenever no ambiguity arises.

\medskip

We shall use the following elementary identities. For $\xi,\zeta\in\Rl$ with $\xi\neq\zeta$ and $n=0,\ldots,M$,
    \begin{enumerate}[label=(S\arabic*),leftmargin=2cm]
        \item 
        \label{prop:symmpoly:SumejX}
            \begin{equation*}
                \sum_{\xi\in\Rl}E_n^{\{\xi\}} = (M-n)E_n.
            \end{equation*}
        \item \label{prop:eN}
            \begin{equation*}
                 E_M = X_\xi^2E_{M-1}^{\{\xi\}} = X_\xi^2X_\zeta^2 E_{M-2}^{\{\xi,\zeta\}}.
            \end{equation*}
        \item \label{prop:symmpoly:eXn}
            \begin{equation*}
                E_n = E_n^{\{\xi\}} + X_\xi^2E_{n-1}^{\{\xi\}}.
            \end{equation*}
    \item \label{prop:symmpoly:ejX=ejkX}
        \begin{equation*}
            E_n^{\{\xi\}} = E_n^{\{\xi,\zeta\}} + X_\zeta^2E_{n-1}^{\{\xi,\zeta\}}.
        \end{equation*}    
    \end{enumerate}
Identity~\ref{prop:symmpoly:SumejX} follows because each monomial in $E_n$ occurs once for every omitted index outside its $n$ factors, hence exactly $M-n$ times. The remaining identities follow by separating the monomials according to whether they contain the indicated variables. In particular,
    \begin{equation*}
        E_{M-1}^{\{\xi\}} = X_\zeta^2E_{M-2}^{\{\xi,\zeta\}}, \qquad \xi\neq\zeta.
    \end{equation*}

\medskip

We identify each particle root $\beta\in R$ with its embedding in $\R^{N+1}$ having zero $e_0$ coordinate. Every interface vector $\alpha\in\Rl$ has $e_0$ coordinate equal to $1$. Consequently,
    \begin{equation*}
        \langle\alpha,\beta\rangle = \sum_{i=1}^N\alpha_i\beta_i, \qquad \alpha\in\Rl,\quad \beta\in R.
    \end{equation*}
    
The reflection associated with $\beta$ is
    \begin{equation*}
        s_\beta(y) = y-\frac{2\langle y,\beta\rangle}{|\beta|^2}\beta, \qquad y\in\R^{N+1}.
    \end{equation*}
It fixes $e_0$ and permutes $\Rl$. Moreover,
    \begin{equation}
        \label{eq:reflect:inner}
        \langle s_\beta(\alpha),\beta\rangle = -\langle\alpha,\beta\rangle,
    \end{equation} 
and
    \begin{align}
        X_\alpha-X_{s_\beta(\alpha)} &\stackrel{\eqref{eq:interface-coordinate}}{=}\langle\alpha-s_\beta(\alpha),\eta\rangle \stackrel{\eqref{eq:extended-loewner-state}}{=}\frac{2\langle\alpha,\beta\rangle}{|\beta|^2}\langle\beta,\lambda\rangle.
        \label{eq:reflect:difference}
    \end{align}
For $\alpha\in\Rl$, define
    \begin{equation}
        \label{def:kappa:b:alpha}
        \kappa_\alpha(t,\lambda) :=\sum_{i=1}^N \alpha_i^2\sigma_i(t,\lambda)^2, \qquad b_\alpha(t,\lambda) :=\sum_{i=1}^N \alpha_i b_i(t,\lambda).
    \end{equation}
In the stochastic calculations below, the particle coefficients are evaluated at $(t,\lambda(t))$, and their arguments are suppressed. We also write $a_\alpha$ for $a_\alpha(t)$ and $X_\alpha$ for $X_\alpha(t,x)$.

\begin{proposition}[Dynamics of the symmetric polynomials]
\label{prop:En:SDE}
For every $n=1,\ldots,M$ and $t<\tau_x\wedge\tau_{\rm coll}$,
    \begin{align}
        dE_n ={}& 2\sum_{i=1}^N\sigma_i \left(\sum_{\alpha\in\Rl}\alpha_iX_\alpha E_{n-1}^{\{\alpha\}}\right)dB_i +2\sum_{\alpha\in\Rl} b_\alpha X_\alpha E_{n-1}^{\{\alpha\}}\,dt + \sum_{\alpha\in\Rl} \kappa_\alpha E_{n-1}^{\{\alpha\}}\,dt\notag\\
            & + 2\sum_{\substack{\alpha,\gamma\in\Rl\\\alpha\neq\gamma}} X_\alpha X_\gamma E_{n-2}^{\{\alpha,\gamma\}}\left(\sum_{i=1}^N\alpha_i\gamma_i\sigma_i^2\right)dt \notag\\
            & + 2\sum_{\beta\in R_+}\frac{k_\beta}{|\beta|^2}\sum_{\substack{\alpha\in\Rl\\s_\beta(\alpha)\neq\alpha}} \langle\alpha,\beta\rangle^2 \left(E_{n-1}^{\{\alpha,s_\beta(\alpha)\}}-X_\alpha X_{s_\beta(\alpha)}E_{n-2}^{\{\alpha,s_\beta(\alpha)\}}\right)dt \notag\\
            & + 2\sum_{\alpha\in\Rl} a_\alpha E_{n-1}^{\{\alpha\}}\,dt + 2\sum_{\beta\in\Rl}\frac{a_\beta}{X_\beta} \sum_{\substack{\alpha\in\Rl\\\alpha\neq\beta}}X_\alpha E_{n-1}^{\{\alpha,\beta\}}\,dt + 2\sum_{\beta\in\Rl}a_\beta X_\beta \sum_{\substack{\alpha\in\Rl\\\alpha\neq\beta}} X_\alpha E_{n-2}^{\{\alpha,\beta\}}\,dt.
            \label{eq:full-interface-polynomial-SDE}
    \end{align}
Sums over distinct interface vectors are taken over ordered pairs. Moreover, for $n,m=1,\ldots,M$,
    \begin{equation}
        d[E_n,E_m]_t = 4\sum_{i=1}^N\sigma_i^2 \left(\sum_{\alpha\in\Rl}\alpha_iX_\alpha E_{n-1}^{\{\alpha\}} \right) 
        \left(\sum_{\gamma\in\Rl}\gamma_iX_\gamma E_{m-1}^{\{\gamma\}}\right)dt.
        \label{eq:full-interface-polynomial-bracket}
    \end{equation}
\end{proposition}
The restriction $s_\beta(\alpha)\neq\alpha$ ensures that the two
deleted interface indices are distinct. Fixed points would contribute
zero because
    \begin{equation*}
        s_\beta(\alpha)=\alpha \quad\Longleftrightarrow\quad \langle\alpha,\beta\rangle=0.
    \end{equation*}
For example, in type $A_{N-1}$,
    \begin{equation*}
        \langle e_0-e_\ell,e_j-e_i\rangle = \delta_{\ell i}-\delta_{\ell j},
    \end{equation*}
which vanishes when $\ell\notin\{i,j\}$.
\begin{proof}
    Since
        \begin{equation*}
            X_\alpha(t,x) {\stackrel{\eqref{eq:interface-coordinate}}{=}} g_t(x)+\sum_{i=1}^N\alpha_i\lambda_i(t),
        \end{equation*}
    the particle {\eqref{eq:general-root-particle-system}} and Loewner equations {\eqref{eq:general-root-adapted-loewner}} give
        \begin{align*}
            dX_\alpha & = \partial_tg_t(x)\,dt + \sum_{i=1}^N\alpha_i\,d\lambda_i \\
                      & = \sum_{\beta\in\Rl}\frac{a_\beta}{X_\beta}\,dt +\sum_{i=1}^N\alpha_i\sigma_i\,dB_i +\sum_{i=1}^N\alpha_i b_i\,dt +\sum_{\beta\in R_+}\frac{k_\beta}{\langle\beta,\lambda\rangle}\left(\sum_{i=1}^N\alpha_i\beta_i\right)dt\\
                      & = \sum_{i=1}^N\alpha_i\sigma_i\,dB_i + b_\alpha\,dt + \sum_{\beta\in R_+}k_\beta \frac{\langle\alpha,\beta\rangle}{\langle\beta,\lambda\rangle}\,dt + \sum_{\beta\in\Rl}\frac{a_\beta}{X_\beta}\,dt.
        \end{align*}
    In the last line, we use the definition of $b_\alpha$ from \eqref{def:kappa:b:alpha}. In particular, independence of the Brownian components yields
        \begin{equation*}
            %\label{eq:interface-quadratic-covariation}
            d[X_\alpha,X_\gamma]_t =\sum_{i=1}^N \alpha_i\gamma_i\sigma_i^2\,dt, \qquad d[X_\alpha]_t=\kappa_\alpha\,dt,
        \end{equation*}
    where we recall the definition of $\kappa_\alpha$ in \eqref{def:kappa:b:alpha}.
    As a polynomial in the interface coordinates, $E_n$, see \eqref{eq:deleted-symmetric-polynomial},satisfies for every $\alpha \in \Rl$ and $\gamma\neq \alpha$
        \begin{equation*}
            \frac{\partial E_n}{\partial X_\alpha} = 2X_\alpha E_{n-1}^{\{\alpha\}}, \qquad \frac{\partial^2 E_n}{\partial X_\alpha^2} = 2E_{n-1}^{\{\alpha\}}, \qquad  \frac{\partial^2 E_n}{\partial X_\alpha\,\partial X_\gamma} = 4X_\alpha X_\gamma E_{n-2}^{\{\alpha,\gamma\}}.
        \end{equation*}
    It\^o's formula therefore gives
        \begin{equation*}
            dE_n = 2\sum_{\alpha\in\Rl} X_\alpha E_{n-1}^{\{\alpha\}}\,dX_\alpha + \sum_{\alpha\in\Rl} E_{n-1}^{\{\alpha\}}\,d[X_\alpha]_t + 2\sum_{\substack{\alpha,\gamma\in\Rl\\\alpha\neq\gamma}} X_\alpha X_\gamma E_{n-2}^{\{\alpha,\gamma\}}\,d[X_\alpha,X_\gamma]_t.
        \end{equation*}
    Substituting the preceding expressions, we obtain
        \begin{align*}
            dE_n ={}& 2\sum_{i=1}^N\sigma_i \left(\sum_{\alpha\in\Rl}\alpha_iX_\alpha E_{n-1}^{\{\alpha\}} \right)dB_i
                + 2\sum_{\alpha\in\Rl} b_\alpha X_\alpha E_{n-1}^{\{\alpha\}}\,dt \\
                & + 2\sum_{\beta\in R_+}\frac{k_\beta}{\langle\beta,\lambda\rangle} \sum_{\alpha\in\Rl} \langle\alpha,\beta\rangle X_\alpha E_{n-1}^{\{\alpha\}}\,dt  + 2\sum_{\beta\in\Rl}\frac{a_\beta}{X_\beta} \sum_{\alpha\in\Rl} X_\alpha E_{n-1}^{\{\alpha\}}\,dt\\
                &+ \sum_{\alpha\in\Rl} \kappa_\alpha E_{n-1}^{\{\alpha\}}\,dt  + 2\sum_{\substack{\alpha,\gamma\in\Rl\\\alpha\neq\gamma}}X_\alpha X_\gamma E_{n-2}^{\{\alpha,\gamma\}}\left(\sum_{i=1}^N \alpha_i\gamma_i\sigma_i^2 \right)dt.
        \end{align*}
    We now simplify the terms in the second line. For the purposes of this simplification only, fix $\beta\in R_+$ and write
    \begin{align}
        \alpha'=s_\beta(\alpha).\label{def:alphaprime}
    \end{align}
    Since $s_\beta$ is an involution permuting $\Rl$, we may pair each summand with its reflected counterpart. Using \eqref{eq:reflect:inner}, this gives
        \begin{align*}
            \sum_{\alpha\in\Rl} \langle\alpha,\beta\rangle X_\alpha E_{n-1}^{\{\alpha\}} & =
            \frac12 \sum_{\substack{\alpha\in\Rl\\\alpha'\neq\alpha}} \left(\langle\alpha,\beta\rangle X_\alpha E_{n-1}^{\{\alpha\}} + \langle\alpha',\beta\rangle X_{\alpha'}E_{n-1}^{\{\alpha'\}} \right)\\
            & = \frac12 \sum_{\substack{\alpha\in\Rl\\\alpha'\neq\alpha}} \langle\alpha,\beta\rangle \left(X_\alpha E_{n-1}^{\{\alpha\}}-X_{\alpha'}E_{n-1}^{\{\alpha'\}} \right).
        \end{align*}
    Fixed points of the reflection contribute zero, since $\alpha'=s_\beta(\alpha)=\alpha$ implies $\langle\alpha,\beta\rangle=0$. For $\alpha'\neq\alpha$, identity~\ref{prop:symmpoly:ejX=ejkX} gives
        \begin{equation*}
            E_{n-1}^{\{\alpha\}}  = E_{n-1}^{\{\alpha,\alpha'\}} + X_{\alpha'}^2 E_{n-2}^{\{\alpha,\alpha'\}}, \qquad
            E_{n-1}^{\{\alpha'\}}  = E_{n-1}^{\{\alpha,\alpha'\}} + X_\alpha^2 E_{n-2}^{\{\alpha,\alpha'\}}.
        \end{equation*}
    Hence,
        \begin{align*}
            X_\alpha E_{n-1}^{\{\alpha\}} -X_{\alpha'}E_{n-1}^{\{\alpha'\}} &= (X_\alpha-X_{\alpha'})E_{n-1}^{\{\alpha,\alpha'\}}+(X_\alpha X_{\alpha'}^2-X_{\alpha'}X_\alpha^2)E_{n-2}^{\{\alpha,\alpha'\}} \\
                & = (X_\alpha-X_{\alpha'}) \left(E_{n-1}^{\{\alpha,\alpha'\}}-X_\alpha X_{\alpha'}E_{n-2}^{\{\alpha,\alpha'\}}\right).
        \end{align*}
        Applying \eqref{eq:reflect:difference}, we conclude that
        \begin{equation*}
            \sum_{\alpha\in\Rl} \langle\alpha,\beta\rangle X_\alpha E_{n-1}^{\{\alpha\}} = \frac{\langle\beta,\lambda\rangle}{|\beta|^2} \sum_{\substack{\alpha\in\Rl\\\alpha'\neq\alpha}} \langle\alpha,\beta\rangle^2 \left(E_{n-1}^{\{\alpha,\alpha'\}}-X_\alpha X_{\alpha'}E_{n-2}^{\{\alpha,\alpha'\}}\right).
        \end{equation*}
    The factor $\langle\beta,\lambda\rangle$ cancels the denominator in the particle-repulsion term, yielding
        \begin{equation*}
            2\sum_{\beta\in R_+}\frac{k_\beta}{\langle\beta,\lambda\rangle}\sum_{\alpha\in\Rl}\langle\alpha,\beta\rangle
            X_\alpha E_{n-1}^{\{\alpha\}} = 2\sum_{\beta\in R_+}\frac{k_\beta}{|\beta|^2} \sum_{\substack{\alpha\in\Rl\\\alpha'\neq\alpha}} \langle\alpha,\beta\rangle^2 \left(E_{n-1}^{\{\alpha,\alpha'\}}-X_\alpha X_{\alpha'}E_{n-2}^{\{\alpha,\alpha'\}}\right)
        \end{equation*}
    which, after using \eqref{def:alphaprime}, is nothing but the third line of \eqref{eq:full-interface-polynomial-SDE}. For the Loewner contribution, first separate the terms with $\alpha=\beta$:
        \begin{equation*}
            \sum_{\beta\in\Rl}\frac{a_\beta}{X_\beta} \sum_{\alpha\in\Rl} X_\alpha E_{n-1}^{\{\alpha\}} = \sum_{\beta\in\Rl} a_\beta E_{n-1}^{\{\beta\}} + \sum_{\beta\in\Rl}\frac{a_\beta}{X_\beta} \sum_{\substack{\alpha\in\Rl\\\alpha\neq\beta}} X_\alpha E_{n-1}^{\{\alpha\}}.
        \end{equation*}
    For $\alpha\neq\beta$, using \ref{prop:symmpoly:ejX=ejkX} gives
        \begin{equation*}
            \sum_{\beta\in\Rl}\frac{a_\beta}{X_\beta} \sum_{\alpha\in\Rl} X_\alpha E_{n-1}^{\{\alpha\}}
            = \sum_{\beta\in\Rl} a_\beta E_{n-1}^{\{\beta\}} + \sum_{\beta\in\Rl}\frac{a_\beta}{X_\beta} \sum_{\substack{\alpha\in\Rl\\\alpha\neq\beta}} X_\alpha E_{n-1}^{\{\alpha,\beta\}} + \sum_{\beta\in\Rl}a_\beta X_\beta \sum_{\substack{\alpha\in\Rl\\\alpha\neq\beta}} X_\alpha E_{n-2}^{\{\alpha,\beta\}}.
        \end{equation*}
    For a positive-weight interface, these calculations are valid because $X_\beta\neq0$ before $\tau_x$. If $a_\beta\equiv0$, its contributions on both sides are zero by convention. Multiplication by the factor $2$ from It\^o's formula gives the last three terms of \eqref{eq:full-interface-polynomial-SDE}. Finally, only the local martingale parts contribute to quadratic covariations, which leads to \eqref{eq:full-interface-polynomial-bracket} and completes the proof.
\end{proof}

\section{Boundary phase transitions in the classical Dyson-Loewner systems}
\label{sec:classical-phase-transitions}

\subsection{Normalized Dyson-Brownian models}

Throughout this section, we assume that $N\geq2$, $\kappa>0$, and that the particle system starts in the interior of its Weyl chamber. We choose
\begin{align*}
    C_A&=\{\lambda\in\R^N:\lambda_1<\cdots<\lambda_N\},\\
    C_B&=\{\lambda\in\R^N:0<\lambda_1<\cdots<\lambda_N\},\\
    C_D&=\{\lambda\in\R^N:|\lambda_1|<\lambda_2<\cdots<\lambda_N\}.
\end{align*}

In \eqref{eq:general-root-particle-system}, we take $\sigma_i=\sqrt{\kappa}$ for $i=1,\ldots,N$, $b_\beta=0$ for every $\beta\in \Rl $ and $k_\alpha=4$ for every $\alpha\in R_+$. The Loewner weights are
    \begin{equation}
        \label{eq:loewner:weights}
        \begin{aligned}
            a_{e_0-e_i}&=2, &&R\in\{A_{N-1},B_N,D_N\},\\
            a_{e_0+e_i}&=2, &&R\in\{B_N,D_N\},\\
            a_{e_0}&=\frac{\kappa}{2}, &&R=B_N.
        \end{aligned}
    \end{equation}
We fix a real initial point $x$ satisfying $X_\alpha(0,x)\neq0$ for every $\alpha\in\mathcal I_R$. The hitting time is
    \begin{equation*}
        \tau_x :=\inf\Big\{ t>0: \liminf_{s\uparrow t}|X_\alpha(s,x)|=0 \text{ for some }\alpha\in\mathcal I_R \Big\}.
    \end{equation*}

Boundary values are understood as limits from within the lifetime of the evolution. Thus, hitting occurs when $g_t(x)$ reaches $\lambda_i(t)$ in type $A_{N-1}$, or $\pm\lambda_i(t)$ in types $B_N$ and $D_N$. In type $B_N$, the origin is also a Loewner singularity. With this normalization, in type $A_{N-1}$ we consider
    \begin{align*}
        %\label{eq:standard-A-dyson} 
        d\lambda_i(t) &= \sqrt{\kappa}\,dB_i(t) + 4\sum_{\substack{j=1\\j\neq i}}^N \frac{dt}{\lambda_i(t)-\lambda_j(t)}, \qquad i=1,\ldots,N, \\ 
        %\label{eq:standard-A-loewner}
        \partial_tg_t(z) &= 2\sum_{i=1}^N \frac{1}{g_t(z)-\lambda_i(t)}. 
    \end{align*}
In type $D_N$, we consider
    \begin{align*}
        %\label{eq:standard-D-dyson}
        d\lambda_i(t) &= \sqrt{\kappa}\,dB_i(t) + 4\sum_{\substack{j=1\\j\neq i}}^N \left(\frac{1}{\lambda_i(t)-\lambda_j(t)}+\frac{1}{\lambda_i(t)+\lambda_j(t)}\right)dt, \qquad i=1,\ldots,N,\\
        %\label{eq:standard-D-loewner}
        \partial_tg_t(z) &= 2\sum_{i=1}^N \left(\frac{1}{g_t(z)-\lambda_i(t)}+\frac{1}{g_t(z)+\lambda_i(t)}\right).
    \end{align*}
Finally, in type $B_N$, we take
    \begin{align}
        %\label{eq:standard-B-dyson}
        d\lambda_i(t) &= \sqrt{\kappa}\,dB_i(t) + \frac{4}{\lambda_i(t)}\,dt  + 4\sum_{\substack{j=1\\j\neq i}}^N \left( \frac{1}{\lambda_i(t)-\lambda_j(t)}+\frac{1}{\lambda_i(t)+\lambda_j(t)}\right)dt, \qquad i=1,\ldots,N,\nonumber\\
\label{eq:standard-B-loewner}      
    \partial_tg_t(z) &= \frac{\kappa}{2g_t(z)} + 2\sum_{i=1}^N\left(\frac{1}{g_t(z)-\lambda_i(t)}+\frac{1}{g_t(z)+\lambda_i(t)}\right).
    \end{align}
In all three cases, the initial condition is $g_0(x)=x\in \R$. 

\medskip

The highest-order polynomial has the explicit form
\begin{equation}
\label{eq:highest-polynomial-three-classical-cases}
E_M(t,x)=
\begin{cases}
    \displaystyle
    \prod_{i=1}^N\bigl(g_t(x)-\lambda_i(t)\bigr)^2,
        &R=A_{N-1},\quad M=N,\\[3mm]
    \displaystyle
    g_t(x)^2
    \prod_{i=1}^N\bigl(g_t(x)^2-\lambda_i(t)^2\bigr)^2,
        &R=B_N,\quad M=2N+1,\\[3mm]
    \displaystyle
    \prod_{i=1}^N\bigl(g_t(x)^2-\lambda_i(t)^2\bigr)^2,
        &R=D_N,\quad M=2N.
\end{cases}
\end{equation}
Put $\zeta_x:=\tau_x\wedge\tau_{\rm coll}$. In type $B_N$, the initial admissibility condition implies $x\neq0$. Set $G_t:=g_t(x)^2$. The Loewner equation gives
    \begin{equation}
        \label{eq:type-B-squared-Loewner-flow}
        \partial_tG_t =\kappa+8G_t\sum_{i=1}^N \frac{1}{G_t-\lambda_i(t)^2}.
    \end{equation}
Before $\tau_{\rm coll}$, all particle coordinates are nonzero. As long as the moving denominators remain nonzero, the right-hand side extends continuously to $G_t=0$, where its  value is $\kappa>0$. A solution starting from $G_0=x^2>0$ therefore cannot first reach zero under these conditions. Consequently, the origin cannot be reached before a moving Loewner singularity and before $\tau_{\rm coll}$. In type $D_N$, the Loewner equation can be written as
    \begin{equation}
        \label{eq:type-D-g-flow}
        \partial_tg_t(x) =4g_t(x)\sum_{i=1}^N \frac{1}{g_t(x)^2-\lambda_i(t)^2}.
    \end{equation}
If $x\neq0$, this equation shows that $g_t(x)$ preserves its sign on $[0,\zeta_x)$. Moreover, it cannot reach zero at a finite time before $\tau_{\rm coll}$ while all moving denominators remain nonzero. It follows that, whenever $E_M(0,x)>0$,
    \begin{equation*}
        %\label{eq:tau-x-highest-polynomial}
        \tau_x\wedge\tau_{\rm coll} = \inf\left\{t>0:\liminf_{s\uparrow t}E_M(s,x)=0 \right\}\wedge\tau_{\rm coll}.
    \end{equation*}
In type $D_N$, the initial point $x=0$ is admissible precisely when
$\lambda_1(0)\neq0$. In this case,
    \begin{equation*}
        g_t(0)=0, \qquad E_{2N}(t,0)=\prod_{i=1}^N\lambda_i(t)^4>0, \qquad t<\zeta_0.
    \end{equation*}
Its hitting time is described in
Proposition~\ref{prop:type-D-central-component}.

\subsection{Boundary phase transitions results}

Here are the main results of the section. We begin with the transition corresponding to $\kappa = 8$, which is just a rescaling of the well-known result for the Dyson particle systems. 
\begin{proposition}
    \label{prop:classical-particle-transition}
    For the classical particle systems specified above, $\tau_{\rm coll}=\infty$ almost surely for $0<\kappa\leq8$ and $\tau_{\rm coll}<\infty$ almost surely for $\kappa>8$.
\end{proposition}

\begin{proof}
    Under the standard normalization and defining $\widehat\lambda(t) = \frac{\lambda(t)}{\sqrt{\kappa}}$, the driving particle system can be written in the common root-system form
        \begin{equation*}
            %\label{eq:rescaled-radial-Dunkl-transition-proof}
            d\widehat\lambda(t) = dB(t) + \frac{4}{\kappa} \sum_{\alpha\in R_+} \frac{\alpha}{\inner{\alpha}{\widehat\lambda(t)} }\,dt,
        \end{equation*}
    which is a radial Dunkl process whose multiplicity at every root is equal to $4/\kappa$. The boundary classification for radial Dunkl processes states that a Weyl-chamber wall is inaccessible when its multiplicity is at least $1/2$, whereas the chamber boundary is reached almost surely in finite time when the multiplicity is strictly smaller than $1/2$; see \cite{Demni2009,Lepingle2010}. This ends the proof. 
\end{proof}
    
The next theorem describes the Loewner-interface transition at $\kappa=4$ in the range where particle collisions are absent. {We present the proof of the following two statements at the end of this section.}

\begin{theorem}
    \label{thm:classical-interface-transition}
    Assume $0<\kappa\leq8$ and $E_M(0,x)>0$. In type $D_N$, assume additionally that $|x|>|\lambda_1(0)|$. Then, $\tau_x=\infty$ almost surely for $0<\kappa\leq4$, whereas $\tau_x<\infty$ almost surely for $4<\kappa\leq8$.
\end{theorem}

For type $D_N$, initial points in the central component
satisfy a different boundary classification.

\begin{proposition}
    \label{prop:type-D-central-component}
    Assume $R=D_N$, $0<\kappa\leq8$, and $|x|<|\lambda_1(0)|$. Then,
        \begin{equation*}
            \tau_x<\infty \quad\text{and}\quad \tau_{\rm coll}=\infty \qquad\text{almost surely}.
        \end{equation*}
    In particular, this conclusion applies to $x=0$.
\end{proposition}
We first establish the squared-Bessel representation and the exhaustion of its intrinsic time, and then prove the two interface results. When $E_M(0,x)>0$, define 

    \begin{equation}
        \label{eq:classical-integrating-factor}
        I_R(t,x):=
            \begin{cases}
                1, &R=A_{N-1},\\[1mm] 
                \displaystyle
                \left|\frac{g_t(x)}{x}\right|^{-2}, &R=B_N,\\[3mm]
                \displaystyle
                \exp\left\{ 2(\kappa-8)\int_0^t\sum_{i=1}^N\frac{du}{g_u(x)^2-\lambda_i(u)^2}\right\},&R=D_N.
            \end{cases}
    \end{equation}
These processes have finite variation, are finite and strictly positive for $t<\zeta_x$, and satisfy $I_R(0,x)=1$. For $R=D_N$ and $x\neq0$, equation \eqref{eq:type-D-g-flow} also gives
    \begin{equation*}
        I_D(t,x) = \left|\frac{g_t(x)}{x}\right|^{(\kappa-8)/2}.
    \end{equation*}

For $0\leq t<\zeta_x$, define the intrinsic time by
    \begin{equation}
        \label{eq:unified-highest-time-change}
        s(t,x) :=
            \begin{cases}
                \displaystyle 
                \kappa \int_0^t I_R(u,x)E_M(u,x)\sum_{i=1}^N\frac{1}{\bigl(g_u(x)-\lambda_i(u)\bigr)^2}\,du,& R=A_{N-1},\\[8mm]
                \displaystyle
                \kappa \int_0^t I_R(u,x)E_M(u,x)\left(4\sum_{i=1}^N\frac{\lambda_i(u)^2}{\bigl(g_u(x)^2-\lambda_i(u)^2\bigr)^2}\right)\,du,& R\in\{B_N,D_N\}.
            \end{cases}
    \end{equation}
In type $A_{N-1}$, we have
    \begin{equation*}
        E_M(t,x) \sum_{i=1}^N \frac{1}{\bigl(g_t(x)-\lambda_i(t)\bigr)^2} = E_{M-1}(t,x),\quad     s(t,x) = \kappa \int_0^t E_{M-1}(u,x)\,du.
    \end{equation*}

%%%%%%%%%%%%% PROPOSITION %%%%%%%%%%%%%%
\begin{proposition}
    \label{thm:classical-highest-BESQ}
    Let $\kappa>0$ and suppose that $E_M(0,x)>0$. Then, almost surely, $s(\cdot,x)$ is continuous and strictly increasing on $[0,\zeta_x)$. On an enlargement of the probability space if necessary, there exists a squared Bessel process $Z$ with starting point $Z_0$ and dimension $\delta$ given by
        \begin{equation*}
            %\label{eq:classical-highest-BESQ-dimension}
            Z_0=E_M(0,x), \qquad \delta=1+\frac4\kappa,
        \end{equation*}
    such that
        \begin{equation}
            \label{eq:classical-highest-BESQ-representation}
            I_R(t,x)E_M(t,x)=Z_{s(t,x)}, \qquad 0\leq t<\zeta_x.
        \end{equation}
    Consequently, $s(\cdot,x)$ maps $[0,\zeta_x)$ bijectively onto $\left[0,\lim_{t\uparrow\zeta_x}s(t,x)\right)$.
\end{proposition}

\begin{proof}
We apply Proposition~\ref{prop:En:SDE} with $n=M$, $\sigma_i=\sqrt{\kappa}$, $b_\alpha=0$, and $k_\beta=4$. All interface coordinates are nonzero for $t<\zeta_x$, and
    \begin{equation*}
        E_{M-1}^{\{\alpha\}} =\frac{E_M}{X_\alpha^2}, \qquad E_{M-1}^{\{\alpha,\gamma\}}=0, \qquad E_{M-2}^{\{\alpha,\gamma\}} =\frac{E_M}{X_\alpha^2X_\gamma^2}, \qquad \alpha\neq\gamma.
    \end{equation*}
Substitution into \eqref{eq:full-interface-polynomial-SDE} gives
    \begin{equation}
        \label{eq:highest-specialization}
        \begin{aligned}
            \frac{dE_M}{E_M} ={}& 2\sqrt{\kappa}\sum_{i=1}^N \left(\sum_{\alpha\in\mathcal I_R}\frac{\alpha_i}{X_\alpha}\right)dB_i +\bigg[\sum_{\alpha\in\mathcal I_R} \frac{\kappa_\alpha+2a_\alpha}{X_\alpha^2} +2\sum_{\substack{\alpha,\gamma\in\mathcal I_R\\
            \alpha\neq\gamma}}\frac{\kappa\sum_{i=1}^N\alpha_i\gamma_i+a_\gamma}{X_\alpha X_\gamma}\bigg]dt\\
            &-8\sum_{\beta\in R_+}\frac{1}{|\beta|^2}\sum_{\substack{\alpha\in\mathcal I_R\\
            s_\beta(\alpha)\neq\alpha}}\frac{\langle\alpha,\beta\rangle^2}{X_\alpha X_{s_\beta(\alpha)}}\,dt.
        \end{aligned}
    \end{equation}
The sums over distinct interface vectors are ordered. In the following, we suppress the time arguments and write $g=g_t(x)$ and $\lambda_i=\lambda_i(t)$.

\bigskip
For type $A_{N-1}$,
    \begin{equation*}
        X_{e_0-e_i}=g-\lambda_i, \qquad \kappa_{e_0-e_i}=\kappa, \qquad a_{e_0-e_i}=2.
    \end{equation*}
For $i\neq j$, the spatial components of $e_0-e_i$ and $e_0-e_j$ have zero inner product. Hence, the off-diagonal covariance contribution in \eqref{eq:highest-specialization} vanishes. The off-diagonal Loewner contribution is
    \begin{equation*}
        4\sum_{i\neq j} \frac{1}{(g-\lambda_i)(g-\lambda_j)} = 8\sum_{i<j}\frac{1}{(g-\lambda_i)(g-\lambda_j)}.
    \end{equation*}
For $\beta=e_j-e_i$, the reflection exchanges $e_0-e_i$ and $e_0-e_j$ and fixes the remaining interface vectors. Since $|\beta|^2=2$ and the two nonzero inner products have square $1$, the contribution of this root is
    \begin{equation*}
        -\frac{8}{2} \left(\frac{1}{(g-\lambda_i)(g-\lambda_j)}+\frac{1}{(g-\lambda_j)(g-\lambda_i)}\right) = -\frac{8}{(g-\lambda_i)(g-\lambda_j)}.
    \end{equation*}
Summing over $i<j$ cancels the off-diagonal Loewner contribution. Consequently,
    \begin{equation}
        \label{eq:type-A-highest-from-general}
        dE_M = -2\sqrt{\kappa}E_M \sum_{i=1}^N\frac{dB_i}{g-\lambda_i} +(\kappa+4)E_M \sum_{i=1}^N\frac{dt}{(g-\lambda_i)^2}.
    \end{equation}
We next evaluate \eqref{eq:highest-specialization} for types $B_N$ and $D_N$. Here, we have \begin{align*}
    \kappa_{e_0\pm e_i}=\kappa, \qquad \kappa_{e_0}=0.
\end{align*}
Put
    \begin{equation*}
        h_i:=g^2-\lambda_i^2.
    \end{equation*}
The relevant identities are
    \begin{align*}
         \frac{1}{g-\lambda_i}+\frac{1}{g+\lambda_i} &=\frac{2g}{h_i},\\
        -\frac{1}{g-\lambda_i}+\frac{1}{g+\lambda_i} &=-\frac{2\lambda_i}{h_i},\\
        \frac{1}{(g-\lambda_i)^2} +\frac{1}{(g+\lambda_i)^2} &=\frac{2(g^2+\lambda_i^2)}{h_i^2}.
    \end{align*}
The martingale term in both cases is therefore
    \begin{equation*}
        -4\sqrt{\kappa}E_M \sum_{i=1}^N\frac{\lambda_i}{h_i}\,dB_i.
    \end{equation*}

Now, we focus on the drift terms in the first line in \eqref{eq:highest-specialization}. For distinct interface vectors, the only nonzero spatial inner products occur between $e_0-e_i$ and $e_0+e_i$, and equal $-1$. The interface $e_0$ in type $B_N$ has zero spatial components. Thus, the diagonal and off-diagonal diffusion contributions to the drift of $dE_M/E_M$ add up to
    \begin{equation*}
        \kappa\sum_{i=1}^N \left(\frac{1}{(g-\lambda_i)^2}+\frac{1}{(g+\lambda_i)^2}\right)-4\kappa\sum_{i=1}^N\frac{1}{h_i}
        =4\kappa\sum_{i=1}^N\frac{\lambda_i^2}{h_i^2}
            -2\kappa\sum_{i=1}^N\frac{1}{h_i}.
\end{equation*}

The Loewner terms in \eqref{eq:highest-specialization} can be combined as
    \begin{equation*}
        2\sum_{\alpha\in\mathcal I_R}\frac{a_\alpha}{X_\alpha^2}
    +2\sum_{\substack{\alpha,\gamma\in\mathcal I_R\\
    \alpha\neq\gamma}}\frac{a_\gamma}{X_\alpha X_\gamma} = 2\left(\sum_{\alpha\in\mathcal I_R}\frac{1}{X_\alpha}\right)\left(\sum_{\gamma\in\mathcal I_R}\frac{a_\gamma}{X_\gamma}\right).
    \end{equation*}
In type $D_N$, plugging in $a_{e_0\pm e_i}=2$ this equals
    \begin{equation*}
        2\left(2g\sum_{i=1}^N\frac{1}{h_i}\right) \left(4g\sum_{i=1}^N\frac{1}{h_i}\right) = 16g^2\left(\sum_{i=1}^N\frac{1}{h_i}\right)^2.
    \end{equation*}
In type $B_N$, using additionally the Loewner weight $a_{e_0}=\kappa/2$ it equals
    \begin{equation*}
        2\left(\frac{1}{g}+2g\sum_{i=1}^N\frac{1}{h_i}\right)\left(\frac{\kappa}{2g}+4g\sum_{i=1}^N\frac{1}{h_i}\right)
        = \frac{\kappa}{g^2}+(8+2\kappa)\sum_{i=1}^N\frac{1}{h_i}+16g^2\left(\sum_{i=1}^N\frac{1}{h_i}\right)^2.
    \end{equation*}
   Next, we analyze the last term in \eqref{eq:highest-specialization} 
For the long-root contribution, fix $i<j$. The reflection associated with $e_j-e_i$ exchanges $e_0-e_i$ with $e_0-e_j$ and $e_0+e_i$ with $e_0+e_j$. The reflection associated with $e_j+e_i$ exchanges $e_0-e_i$ with $e_0+e_j$ and $e_0+e_i$ with $e_0-e_j$. Each exchanged pair is counted twice in the interface sum. Both roots have squared norm $2$, and every nonzero $\langle\alpha,\beta\rangle^2$ equals $1$. Their combined contribution is therefore
    \begin{equation*}
        -8\bigg[\frac{1}{(g-\lambda_i)(g-\lambda_j)}+\frac{1}{(g+\lambda_i)(g+\lambda_j)}+\frac{1}{(g-\lambda_i)(g+\lambda_j)}+\frac{1}{(g+\lambda_i)(g-\lambda_j)}\bigg] = -\frac{32g^2}{h_i h_j}.
    \end{equation*}
Thus, the total long-root contribution is
    \begin{equation*}
        -32g^2\sum_{i<j}\frac{1}{h_i h_j}.
    \end{equation*}
Together with the quadratic Loewner term, this gives
    \begin{equation*}
        16g^2\left(\sum_{i=1}^N\frac{1}{h_i}\right)^2-32g^2\sum_{i<j}\frac{1}{h_i h_j} = 16g^2\sum_{i=1}^N\frac{1}{h_i^2}.
    \end{equation*}
    In type $B_N$, there is also the short root $\beta=e_i$. Its reflection exchanges $e_0-e_i$ and $e_0+e_i$, while fixing $e_0$ and the other interface vectors. Since $|\beta|^2=1$, its contribution is
    \begin{equation*}
        -8\left(\frac{1}{(g-\lambda_i)(g+\lambda_i)}+\frac{1}{(g+\lambda_i)(g-\lambda_i)}\right) =-\frac{16}{h_i}.
    \end{equation*}
Summing these terms gives $-16\sum_i h_i^{-1}$. Finally, using
    \begin{equation*}
        \frac{g^2}{h_i^2}=\frac{1}{h_i}+\frac{\lambda_i^2}{h_i^2},
    \end{equation*}
we obtain in type $B_N$
    \begin{equation}
        \label{eq:type-B-highest-from-general}
        dE_M=-4\sqrt{\kappa}E_M\sum_{i=1}^N\frac{\lambda_i}{h_i}\,dB_i+4(\kappa+4)E_M\sum_{i=1}^N\frac{\lambda_i^2}{h_i^2}\,dt
        +E_M\left(\frac{\kappa}{g^2}+8\sum_{i=1}^N\frac{1}{h_i}\right)dt,
    \end{equation}
and in type $D_N$
    \begin{equation}
        \label{eq:type-D-highest-from-general}
        dE_M = -4\sqrt{\kappa}E_M\sum_{i=1}^N\frac{\lambda_i}{h_i}\,dB_i+4(\kappa+4)E_M\sum_{i=1}^N\frac{\lambda_i^2}{h_i^2}\,dt+2(8-\kappa)E_M\sum_{i=1}^N\frac{1}{h_i}\,dt.
    \end{equation}
The integrating factors in \eqref{eq:classical-integrating-factor}
satisfy 
    \begin{align*}
        \frac{dI_B}{I_B} &\stackrel{\eqref{eq:standard-B-loewner}}{=}-\left(\frac{\kappa}{g^2}+8\sum_{i=1}^N\frac{1}{h_i}\right)dt,\\
        \frac{dI_D}{I_D} &=2(\kappa-8)\sum_{i=1}^N\frac{1}{h_i}\,dt.
    \end{align*}
Since these factors have finite variation, the product rule shows that the additional drift terms cancel in $Y_t:=I_R(t,x)E_M(t,x)$. In all three cases, the drift of $Y$ is $(1+4/\kappa)\,ds(t,x)$, and $d[Y]_t=4Y_t\,ds(t,x)$. The time density is strictly positive. For type $A_{N-1}$, this follows from positivity of the inverse squared gaps. For types $B_N$ and $D_N$,
    \begin{equation*}
        \sum_{i=1}^N \frac{\lambda_i(t)^2}{(g_t(x)^2-\lambda_i(t)^2)^2}>0,
    \end{equation*}
because the particle vector is nonzero in the chamber interior. Together with $Y_t>0$, this proves that $s(\cdot,x)$ is continuous and strictly increasing. Let $t(u)$ be the inverse time and put $Z_u=Y_{t(u)}$. The time-change theorem gives
    \begin{equation*}
        dZ_u = 2\sqrt{Z_u}\,d\widetilde W_u+\left(1+\frac{4}{\kappa}\right)du
    \end{equation*}
up to the terminal time value, for a one-dimensional Brownian motion $\widetilde W$. Thus $Z$ is a squared Bessel process of the stated dimension, started from $Y_0=E_M(0,x)$. If the terminal time value is finite, $Z$ can be extended beyond it as a squared Bessel process, on an enlargement of the probability space if necessary. This proves the representation.
\end{proof}

The squared-Bessel representation is expressed in terms of the intrinsic time $s(t,x)$. To obtain conclusions in the original Loewner time, we must show that this time diverges on trajectories with infinite lifetime.

\begin{lemma}
    \label{lem:classical-time-exhaustion}
    Assume $0<\kappa\leq8$ and $E_M(0,x)>0$. Then,
    \begin{equation*}
        %\label{eq:classical-time-exhaustion}
        \lim_{t\to\infty}s(t,x)=\infty
    \end{equation*}
    almost surely on $\{\zeta_x=\infty\}.$
\end{lemma}
\begin{proof}
    By Proposition~\ref{prop:classical-particle-transition}, $\tau_{\rm coll}=\infty$ almost surely. Work on $\{\zeta_x=\infty\}$ and write
        \begin{equation}
            \label{eq:def:Z}
            Y_t:=I_R(t,x)E_M(t,x)=Z_{s(t,x)}, \qquad p:=\frac{N-1}{N}<1.
        \end{equation}
    Suppose, by contradiction, that
        \begin{equation}
            \label{eq:finite-time-assumption}
            S:=\lim_{t\to\infty}s(t,x)<\infty
        \end{equation}
    on an event of positive probability. We first obtain an integrability consequence of this assumption from the squared-Bessel representation. For type $A_{N-1}$, Maclaurin's inequality then gives an immediate contradiction. For types $B_N$ and $D_N$, we derive a finite weighted integral of the particle coordinates and contradict it using the radial process and a negative-moment estimate.

    Let $T_0(Z):=\inf\{u>0:Z_u=0\}$. Since $Y_t>0$ for every finite $t$ and the time has range $[0,S)$, we have $Z_u>0$ for $u<S$, and hence $S\leq T_0(Z)$. We claim that
        \begin{equation}
            \label{eq:negative-power-integrability}
            \int_0^S Z_u^{-p}\,du<\infty \qquad\text{almost surely}.
        \end{equation}
    If $S<T_0(Z)$, this follows from continuity and positivity on $[0,S]$. If $S=T_0(Z)<\infty$, the squared-Bessel dimension $\delta=1+4/\kappa$ is smaller than $2$. Choose
        \begin{equation*}
            0<r<\min\{1-p,1-\delta/2\}.
        \end{equation*}
    Before $T_0(Z)$, It\^o's formula gives
        \begin{equation*}
            dZ_u^r = 2rZ_u^{r-1/2}\,d\widetilde W_u-r(2-\delta-2r)Z_u^{r-1}\,du.
        \end{equation*}
    For $0<\varepsilon<Z_0<K$, stop at the first exit $\sigma_{\varepsilon,K}$ from $(\varepsilon,K)$ and at a deterministic time $m$. Taking expectations yields
        \begin{equation*}
             r(2-\delta-2r)\mathbb E\left[\int_0^{m\wedge\sigma_{\varepsilon,K}}Z_u^{r-1}\,du\right] \leq Z_0^r.
        \end{equation*}
    Letting $\varepsilon\downarrow0$ and $m\uparrow\infty$ shows that the integral of $Z^{r-1}$ is finite up to $T_0(Z)\wedge T_K$, where $T_K:=\inf\{u>0:Z_u=K\}$. On $\{T_0(Z)<\infty\}$, continuity makes the path bounded on $[0,T_0(Z)]$, so a countable choice of $K$ gives
        \begin{equation*}
            \int_0^{T_0(Z)}Z_u^{r-1}\,du<\infty.
        \end{equation*}
    Since $1-p-r>0$ and $Z_u^{-p}=Z_u^{r-1}Z_u^{1-p-r}$, this proves \eqref{eq:negative-power-integrability}.

    \medskip

    In type $A_{N-1}$, $I_R=1$ and $M=N$. Maclaurin's inequality gives
        \begin{equation*}
            E_{N-1}(t,x)\geq N E_N(t,x)^{(N-1)/N} \stackrel{\eqref{eq:def:Z}}{=} NY_t^p=NZ_{s(t,x)}^p.
        \end{equation*}
    Since $ds(t,x)=\kappa E_{N-1}(t,x)\,dt$, see \eqref{eq:unified-highest-time-change}, we obtain
        \begin{equation*}
            t\leq\frac1{\kappa N} \int_0^tY_u^{-p}\,ds(u,x) = \frac1{\kappa N}\int_0^{s(t,x)}Z_v^{-p}\,dv.
        \end{equation*}
    Letting $t\to\infty$ contradicts \eqref{eq:negative-power-integrability}. 
    \medskip
    
    For types $B_N$ and $D_N$, set
        \begin{equation*}
            L(t):=\prod_{i=1}^N\lambda_i(t)^2, \qquad Q(t):=\sum_{i=1}^N\lambda_i(t)^2.
        \end{equation*}
    The arithmetic-geometric mean inequality and \eqref{eq:highest-polynomial-three-classical-cases} give
        \begin{equation*}
            \sum_{i=1}^N\frac{\lambda_i(t)^2}{(g_t(x)^2-\lambda_i(t)^2)^2} \geq
                \begin{cases}
                    \displaystyle
                    N\left(\frac{g_t(x)^2L(t)}{E_M(t,x)}\right)^{1/N}, &R=B_N,\\[3mm]
                    \displaystyle N\left(\frac{L(t)}{E_M(t,x)}\right)^{1/N}, &R=D_N.
                \end{cases}
        \end{equation*}
    Recall the time change \eqref{eq:unified-highest-time-change}, and using \eqref{eq:def:Z} as well as $I_B(t,x)g_t(x)^2=x^2$, see \eqref{eq:classical-integrating-factor}, we obtain
        \begin{equation*}
             \partial_t s(t,x) \geq 
             \begin{cases}
                 4\kappa N Y_t^p \big(x^2L(t)\big)^{1/N}, & R=B_N,  \\
                 4\kappa N Y_t^p \bigl(I_D(t,x)L(t)\bigr)^{1/N},  & R=D_N. 
             \end{cases}
        \end{equation*}
Integrating against $Y_t^{-p}$ and changing variables to the intrinsic time gives
        \begin{equation}
            \label{eq:finite-weighted-integral}
            |x|^{2/N}\int_0^\infty L(t)^{1/N}\, dt\leq\ \frac{1}{4\kappa N}\int_0^S Z_u^{-p}\,du <\infty.
        \end{equation}
    for $B_N$ and similarly in the $D_N$ case it leads to
        \begin{equation}
             \label{eq:finite-weighted-integral-2}
             \int_0^\infty\bigl(I_D(t,x)L(t)\bigr)^{1/N}\,dt\leq\ \frac{1}{4\kappa N}\int_0^S Z_u^{-p}\,du <\infty.
        \end{equation}
     We next express this consequence in terms of the particle process alone. Put
        \begin{equation*}
            q:=
                \begin{cases}
                        0,&R=B_N,\\
                    (\kappa-8)/2,&R=D_N.
                \end{cases}
        \end{equation*}
    Thus $-4<q\leq0$. We shall show that, for a finite random time $t_1\geq1$,
        \begin{equation}
            \label{eq:remaining-target-integral-finite}
            \int_{t_1}^{\infty} L(t)^{1/N}Q(t)^{q/(2N)}\,dt<\infty.
        \end{equation}
    In type $B_N$, this follows directly from \eqref{eq:finite-weighted-integral}, because $x\neq0$ and $q=0$. For type $D_N$, first suppose that $x\neq0$. Then
        \begin{equation*}
            I_D(t,x)=|x|^{-q}|g_t(x)|^q, \qquad Y_t=|x|^{-q}|g_t(x)|^q \prod_{i=1}^N \bigl(g_t(x)^2-\lambda_i(t)^2\bigr)^2.
        \end{equation*}
    Assumption \eqref{eq:finite-time-assumption} and continuity of $Z$ imply $\sup_{t\geq0}Y_t<\infty$. Further, to get lower bounds for $Y_t$, we consider two cases for $g_t(x)^2$. Whenever $g_t(x)^2>2Q(t)$,
        \begin{equation*}
            g_t(x)^2-\lambda_i(t)^2\geq\frac12 g_t(x)^2, \qquad i=1,\ldots,N,
        \end{equation*}
    and therefore
        \begin{equation*}
            Y_t\geq 2^{-2N}|x|^{-q}|g_t(x)|^{4N+q}.
        \end{equation*}
    Since $4N+q>0$, this bounds $g_t(x)^2$ on that region.  
    On its complement, $g_t(x)^2\leq2Q(t)$. Consequently, there is a finite random constant $C>0$ such that
        \begin{equation*}
            g_t(x)^2\leq C(1+Q(t)), \qquad t\geq0.
        \end{equation*}
    Because $q\leq0$, for $x\neq 0$ it follows that
        \begin{align*}
            I_D(t,x)^{1/N} &=|x|^{-q/N}\bigl(g_t(x)^2\bigr)^{q/(2N)}\\
                &\geq |x|^{-q/N}C^{q/(2N)}(1+Q(t))^{q/(2N)}.
        \end{align*}
    If $x=0$, then $g_t(0)=0$ and
        \begin{equation*}
            I_D(t,0) =\exp\left\{2(8-\kappa)\int_0^t \sum_{i=1}^N\frac{du}{\lambda_i(u)^2}\right\} \geq1.
        \end{equation*}
    Since $(1+Q(t))^{q/(2N)}\leq1$, the same lower bound holds, with a suitable positive constant. Thus, in both cases
        \begin{equation*}
            I_D(t,x)^{1/N} \geq c(1+Q(t))^{q/(2N)}
        \end{equation*}
    for some finite random constant $c>0$.
    
    In both types, It\^o's formula gives
        \begin{equation*}
            %\label{eq:remaining-radial-BESQ-SDE}
            dQ(t)=2\sqrt{\kappa Q(t)}\,dW^Q(t) +\kappa\delta_Q\,dt, \qquad \delta_Q=
                \begin{cases}
                    \displaystyle N+\frac{8N^2}{\kappa}, &R=B_N,\\[2mm]
                    \displaystyle N+\frac{8N(N-1)}{\kappa}, &R=D_N.
                \end{cases}
        \end{equation*}
    Here, each pair of long-root interactions contributes $16$ to the drift of $Q$, while the short roots in type $B_N$ contribute an additional $8N$. Thus $Q/\kappa$ is a squared Bessel process of dimension $\delta_Q>2$, and $Q(t)\to\infty$ almost surely. Choose a finite random $t_1\geq1$ such that $Q(t)\geq1$ for $t\geq t_1$. Since $1+Q(t)\leq2Q(t)$ on this interval, the preceding lower bound and \eqref{eq:finite-weighted-integral-2} prove \eqref{eq:remaining-target-integral-finite} also in type $D_N$. Furthermore,
        \begin{equation}
            \label{eq:remaining-inverse-Q-divergence}
            \int_{t_1}^\infty\frac{dt}{Q(t)}=\infty \qquad\text{almost surely}.
        \end{equation}
    Indeed,
        \begin{equation*}
            d\log Q(t) = \frac{2\sqrt{\kappa}}{\sqrt{Q(t)}}\,dW^Q(t) + \frac{\kappa(\delta_Q-2)}{Q(t)}\,dt.
        \end{equation*}
    If the integral in \eqref{eq:remaining-inverse-Q-divergence} were finite, the local martingale would have finite total quadratic variation and hence a finite limit. The finite-variation term would also converge. This would force $\log Q(t)$ to have a finite limit, contradicting $Q(t)\to\infty$.

    \medskip

    Choose $\theta\in(0,1)$ sufficiently small and set
        \begin{equation*}
            \gamma:=\frac{\theta}{N(1-\theta)}, \qquad \beta:=\frac{2+q\theta/N}{2-2\theta}, \qquad c_\theta:=N\gamma+\beta.
        \end{equation*}
    Since $N\geq2$ and $-4<q\leq0$,
        \begin{equation*}
            c_\theta = \frac{1+(1+q/(2N))\theta}{1-\theta}>1.
        \end{equation*}
    Moreover, $\gamma\to0$ and $c_\theta\to1$ as $\theta\downarrow0$. We may therefore require
        \begin{equation}
            \label{eq:remaining-holder-conditions}
            2\gamma<1, \qquad 2c_\theta<\delta_Q.
        \end{equation}
    We next establish
        \begin{equation}
            \label{eq:remaining-negative-moment-estimate}
            \mathbb E\left[L(t)^{-\gamma}Q(t)^{-\beta}\right] \leq Ct^{-c_\theta}, \qquad t\geq1,
        \end{equation}
    for a fixed initial particle configuration. Define
        \begin{equation*}
            H(y):= \left(\prod_{i=1}^N y_i^2\right)^{-\gamma} \left(\sum_{i=1}^N y_i^2\right)^{-\beta}.
        \end{equation*}
    This function is homogeneous of degree $-2c_\theta$. Self-similarity gives
        \begin{equation*}
            \mathbb E_{\lambda(0)}[H(\lambda(t))] = t^{-c_\theta} \mathbb E_{\lambda(0)/\sqrt t}[H(\lambda(1))].
        \end{equation*}
    For starting points $z$ in a compact subset of the closed chamber, the transition-density formula for the radial Dunkl process and the exponential bound for the Dunkl kernel imply
        \begin{equation*}
            p_1(z,y) \leq C_1e^{-c_1|y|^2} \prod_{\alpha\in R_+} |\langle\alpha,y\rangle|^{8/\kappa}, \qquad y\in C_R;
        \end{equation*}
    see \cite[Section~3]{Chybiryakov2008}. The constants may depend on the compact set, $R$, and $\kappa$. The root weight is bounded on the unit sphere. The remaining angular singularities are bounded by a constant times $\prod_i|y_i|^{-2\gamma}$, which is integrable on the sphere because $2\gamma<1$. In polar coordinates, the radial power near the origin is $r^{\delta_Q-1-2c_\theta}$, which is integrable by \eqref{eq:remaining-holder-conditions}. The Gaussian factor gives integrability at infinity. Thus. $\mathbb E_z[H(\lambda(1))]$ is uniformly bounded for $z=\lambda(0)/\sqrt t$, $t\geq1$, proving \eqref{eq:remaining-negative-moment-estimate}.

    Since $c_\theta>1$, Tonelli's theorem gives
        \begin{equation}
            \label{eq:remaining-complementary-integral-finite}
            \int_1^\infty L(t)^{-\gamma}Q(t)^{-\beta}\,dt<\infty \qquad\text{almost surely}.
        \end{equation}
    In type $D_N$, a coordinate may cross zero. However, each coordinate has quadratic variation $\kappa\,dt$, so the occupation-time formula gives the Lebesgue measure
        \begin{equation*}
            \operatorname{Leb}\{t\geq0:L(t)=0\}=0 \qquad\text{almost surely}.
        \end{equation*}
    In type $B_N$, all coordinates remain strictly positive. Consequently, in both cases, for almost every $t$,
    
        \begin{equation*}
            \frac{1}{Q(t)} = \bigl(L(t)^{1/N}Q(t)^{q/(2N)}\bigr)^\theta \bigl(L(t)^{-\gamma}Q(t)^{-\beta}\bigr)^{1-\theta}.
        \end{equation*}
    H\"older's inequality therefore yields
        \begin{equation*}
            \int_{t_1}^T\frac{dt}{Q(t)}\leq  \left(\int_{t_1}^T L(t)^{1/N}Q(t)^{q/(2N)}\,dt \right)^\theta \left(\int_{t_1}^T L(t)^{-\gamma}Q(t)^{-\beta}\,dt\right)^{1-\theta}.
        \end{equation*}
    Letting $T\to\infty$, equations \eqref{eq:remaining-target-integral-finite} and \eqref{eq:remaining-complementary-integral-finite} would imply
        \begin{equation*}
            \int_{t_1}^\infty\frac{dt}{Q(t)}<\infty,
        \end{equation*}
    contradicting \eqref{eq:remaining-inverse-Q-divergence}. Thus, the finite-time event has probability zero.
\end{proof}
\begin{proof}[Proof of Theorem~\ref{thm:classical-interface-transition}]
    By Proposition~\ref{prop:classical-particle-transition}, $\tau_{\rm coll}=\infty$ almost surely throughout the range $0<\kappa\leq8$. Let $Z$ be the squared Bessel process from Proposition~\ref{thm:classical-highest-BESQ}, and let $T_0(Z):=\inf\{u>0:Z_u=0\}$. Its dimension is $\delta=1+4/\kappa$. The boundary classification of squared Bessel processes gives
        \begin{equation}
            \label{eq:interface-BESQ-zero-classification}
            \begin{cases}
                T_0(Z)=\infty \quad\text{almost surely}, &0<\kappa\leq4,\\
                T_0(Z)<\infty \quad\text{almost surely}, &4<\kappa\leq8;
            \end{cases}
        \end{equation}
    see \cite[Chapter~XI]{RevuzYor}.

    \medskip

    First suppose $0<\kappa\leq4$, and assume, by contradiction, that $\tau_x<\infty$ on an event of positive probability. The particle paths remain finite and in the chamber interior at this time. At a finite hitting time, the Loewner image has a terminal value at one of the poles. There is no finite-time escape to infinity, since the particles are bounded on compact time intervals and the Loewner vector field is of order $1/g$ for large $|g|$. In type $B_N$, the origin cannot be reached before a moving pole, as shown by \eqref{eq:type-B-squared-Loewner-flow}. All moving poles are nonzero at $\tau_x$, since $\tau_{\rm coll}=\infty$. Hence, $|g_t(x)|$ has a finite, strictly positive limit as $t\uparrow\tau_x$, and so does $I_B(t,x)$. We next consider type $D_N$. The noncentral initial condition and continuity imply
        \begin{equation}
            \label{eq:type-D-noncentral-gap-positive}
            g_t(x)^2-\lambda_1(t)^2>0, \qquad 0\leq t<\tau_x.
        \end{equation}
    Suppose that the first interface hit occurred at the origin. Then
        \begin{equation*}
            g_t(x)^2\longrightarrow0, \qquad \lambda_1(t)^2\longrightarrow0 \qquad\text{as }t\uparrow\tau_x.
        \end{equation*}
    Every other particle coordinate stays bounded away from zero near $\tau_x$, because the particle configuration remains in the chamber interior. On the other hand,
        \begin{equation}
            \label{eq:type-D-h-flow-interface-proof}
            \partial_tg_t(x)^2 =8\sum_{i=1}^N \frac{g_t(x)^2}{g_t(x)^2-\lambda_i(t)^2}.
        \end{equation}
    By \eqref{eq:type-D-noncentral-gap-positive},
        \begin{equation*}
            \frac{g_t(x)^2}{g_t(x)^2-\lambda_1(t)^2}\geq1,
        \end{equation*}
    whereas all terms with $i\geq2$ tend to zero. Thus, the derivative in \eqref{eq:type-D-h-flow-interface-proof} is bounded below by a positive constant near $\tau_x$. This is incompatible with the positive function $g_t(x)^2$ tending to zero. The first interface hit therefore occurs away from the origin. Consequently,
        \begin{equation*}
            I_D(t,x) = \left|\frac{g_t(x)}{x}\right|^{(\kappa-8)/2}
        \end{equation*}
    also has a finite, strictly positive terminal limit. In all three cases, $E_M(t,x)\to0$ as $t\uparrow\tau_x$, and $I_R(t,x)$ has a finite, strictly positive limit. Recall \eqref{eq:highest-polynomial-three-classical-cases} and \eqref{eq:unified-highest-time-change} the time density can be written as
        \begin{equation*}
            %\label{eq:terminal-time-density-interface-proof}
            \frac{1}{\kappa}\partial_t s(t,x) =
                \begin{cases}
                    E_{N-1}(t,x), &R=A_{N-1},\\[2mm]
                    \displaystyle
                    4I_B(t,x)g_t(x)^2 \sum_{i=1}^N\lambda_i(t)^2 \prod_{j\neq i} \bigl(g_t(x)^2-\lambda_j(t)^2\bigr)^2, &R=B_N,\\[3mm]
                    \displaystyle
                    4I_D(t,x) \sum_{i=1}^N\lambda_i(t)^2 \prod_{j\neq i}\bigl(g_t(x)^2-\lambda_j(t)^2\bigr)^2, &R=D_N.
                \end{cases}
        \end{equation*}
    This expression is bounded near $\tau_x$, so
        \begin{equation*}
            S_x:=\lim_{t\uparrow\tau_x}s(t,x)<\infty.
        \end{equation*}
    Furthermore,
        \begin{equation*}
             I_R(t,x)E_M(t,x)\longrightarrow0.
        \end{equation*}
    The representation \eqref{eq:classical-highest-BESQ-representation} and continuity of $Z$ imply $Z_{S_x}=0$, contradicting \eqref{eq:interface-BESQ-zero-classification}. This proves $\tau_x=\infty$ almost surely for $0<\kappa\leq4$.

    \medskip

    Now, suppose $4<\kappa\leq8$. Then $T_0(Z)<\infty$ almost surely. On the event $\{\tau_x=\infty\}$, we have $\zeta_x=\infty$, and Lemma~\ref{lem:classical-time-exhaustion} gives   
        \begin{equation*}
            \lim_{t\to\infty}s(t,x)=\infty.
        \end{equation*}
    The continuous, strictly increasing time must therefore reach $T_0(Z)$ at a finite time $t_0$. The squared-Bessel representation then gives
        \begin{equation*}
            I_R(t_0,x)E_M(t_0,x)=Z_{T_0(Z)}=0.
        \end{equation*}
    This contradicts positivity of both factors for $t_0<\zeta_x$. Hence, $\tau_x<\infty$ almost surely.
\end{proof}

\begin{proof}[Proof of Proposition~\ref{prop:type-D-central-component}]
    By Proposition~\ref{prop:classical-particle-transition}, $\tau_{\rm coll}=\infty$ almost surely. In the chosen type-$D_N$ chamber, the first particle satisfies
        \begin{equation*}
           % \label{eq:type-D-smallest-coordinate-central-proof}
            d\lambda_1(t) =\sqrt{\kappa}\,dB_1(t) + 8\lambda_1(t) \sum_{j=2}^N \frac1{\lambda_1(t)^2-\lambda_j(t)^2}\,dt.
        \end{equation*}
    Define
        \begin{equation*}
            J(t):=\exp\left\{-8\int_0^t\sum_{j=2}^N \frac1{\lambda_1(u)^2-\lambda_j(u)^2}\,du\right\}.
        \end{equation*}
    The process $J$ is finite on compact time intervals, has finite variation, and satisfies $J(t)\geq1$. The product rule gives
        \begin{equation*}
            d\bigl(J(t)\lambda_1(t)\bigr) = \sqrt{\kappa}J(t)\,dB_1(t),
        \end{equation*}
    and hence
        \begin{equation*}
            [J\lambda_1]_t = \kappa\int_0^tJ(u)^2\,du \geq\kappa t.
        \end{equation*}
    By the Dambis-Dubins-Schwarz theorem, there exists a one-dimensional Brownian motion $\widetilde B$ such that
        \begin{equation*}
            J(t)\lambda_1(t) = \lambda_1(0)+\widetilde B_{[J\lambda_1]_t}.
        \end{equation*}
    Since the quadratic variation is continuous and tends to infinity, this process reaches zero in finite time almost surely. As $J(t)>0$, it follows that $\rho_0:=\inf\{t>0:\lambda_1(t)=0\}<\infty$ almost surely. The central initial condition and continuity imply
        \begin{equation*}
            |g_t(x)|<|\lambda_1(t)|, \qquad 0\leq t<\tau_x\wedge\rho_0.
        \end{equation*}
    If $\rho_0<\tau_x$, continuity at $\rho_0$ would give $g_{\rho_0}(x)=\lambda_1(\rho_0)=0$,  which is itself a moving-interface hit. Therefore, $\tau_x\leq\rho_0<\infty$ almost surely.  For $x=0$, equation \eqref{eq:type-D-g-flow} gives $g_t(0)=0$ for $t<\rho_0$. No moving pole is zero before $\rho_0$, so in this case $\tau_0=\rho_0$. Together with $\tau_{\rm coll}=\infty$, this proves the proposition.
\end{proof}

\section{Study of hitting times when in-between drivers}\label{sec:collisions}
The behavior of $E_M(t,x)$ gives us information on whether hitting occurs when $0<\kappa\leq8$ on $t\in[0,\zeta_x)$, all of which was presented in the previous section.
The advantage of that approach was that $E_M(t,x)$ cancels if and only if a hitting event occurs, but the disadvantage lies in that it does not make a distinction for which driver hits $g_t(x)$.
This is clearly not a problem if, for example, $x>\lambda_N(0)$, namely, $x$ lies to the right of all drivers.
In this case, because drivers' paths are continuous, it is evident that if any driver hits $g_t(x)$, it must be $\lambda_N$, the rightmost driver.

However, it is also of interest to consider which driver hit $g_t(x)$ if $x$ is in between two drivers, for example $\lambda_i(0)<x<\lambda_{i+1}(0)$.
This situation requires a slightly different approach: we clearly require a probe that can distinguish one driver from the other.
The distinction cannot be made by using symmetric polynomials, as they treat all drivers in the same footing by definition.

The strategy we present below is based on defining an asymmetric rational function which, at time $\tau_x$, takes the value zero to indicate the left-neighbor driver, and one to indicate the right-neighbor driver.
We consider the Dyson case with two drivers ($A_{1}$, $N=2$) as the simplest non-trivial example of this situation.
We suspect that this strategy may be applied in broader generality; however, computations are significantly more involved when handling three or more drivers, or in the $B_N$ and $D_N$ cases.
Nevertheless, we expect the results below to be significant because if $x$ starts between two drivers, these two drivers are the only ones that can hit $g_t(x)$ and have the greatest influence on its evolution.
In fact, we can see from \eqref{eq:multipleSLE} that the evolution of $g_t(x)$ near a driver is dominated by the term in the Loewner equation involving that driver, so it could that the insights obtained from this two-driver setting can be generalized to the full generality presented in Section~\ref{sec:classical-phase-transitions}.

Let us begin our exploration here. Set $N=2$ and $\lambda_1(0)<x<\lambda_2(0)$. Our aim is to reduce the problem to a one-dimensional process 
\begin{align*}
    U(t):=\frac{g_t(x)-\lambda_1(t)}{\lambda_2(t)-\lambda_1(t)}
\end{align*}
hitting its boundary. Since the starting point is between the two drivers, we have $U(t)\in(0,1)$ for $t\in[0,\zeta_x)$. In this particular case $\tau_x$ is given by
\begin{align*}
\tau_x=\inf\{t>0:\ g_t(x)=\lambda_1(t)\ \text{or}\ g_t(x)=\lambda_2(t)\},
\end{align*}
and we stress that we are always working with time up to $\zeta_x$. It is then clear that hitting one of the drivers corresponds to $U(t)=0$ or $U(t)=1$. Our first step is to calculate the SDE of $U$.
\begin{lemma} \label{lm:SDE:U}
    Define $Y:=\lambda_2-\lambda_1$. For $\kappa>0$ and $t\in[0,\zeta_x)$, the process $U$ satisfies the SDE 
    \begin{align*}
    dU(t)
    &=\frac{\sqrt{\kappa}}{Y(t)}\sqrt{U(t)^2+\bigl(1-U(t)\bigr)^2}\,dW(t)+\frac{2U(t)-1}{Y(t)^2}\left[\frac{2}{U(t)\bigl(U(t)-1\bigr)}+(\kappa-4)\right]dt
    \end{align*}
    with
    \[dW(t):=-\frac{\bigl(1-U(t)\bigr)\,dB_1(t)+U(t)\,dB_2(t)}{\sqrt{U(t)^2+\bigl(1-U(t)\bigr)^2}}.\]
\end{lemma}
\begin{proof}
    As in previous derivations, we omit the time-variable in parentheses when possible. For $t\in[0,\zeta_x)$ we define $X_1$ by
    \begin{align*} 
        X_1:=g_t(x)-\lambda_1,
    \end{align*}
    and realize that $U=X_1/Y$. We also note that
    \begin{align*}
        g_t(x)-\lambda_1=X_1=UY,\qquad 
        g_t(x)-\lambda_2=X_1-Y=\big(U-1\big)Y.
    \end{align*}
    Inserting this into \eqref{eq:sde:dyson} and \eqref{eq:multipleSLE} yields
    \begin{align*}
        d g_t(x)&=\left(\frac{2}{g_t(x)-\lambda_1}+\frac{2}{g_t(x)-\lambda_2}\right) dt=\left(\frac{2}{X_1}+\frac{2}{X_1-Y}\right)dt
\\
        d\lambda_1&=\sqrt{\kappa}\,dB_1+\frac{4}{\lambda_1-\lambda_2}\,dt
        =\sqrt{\kappa}\,dB_1-\frac{4}{Y}\,dt,\\
        d\lambda_2&=\sqrt{\kappa}\,dB_2+\frac{4}{\lambda_2-\lambda_1}\,dt
        =\sqrt{\kappa}\,dB_2+\frac{4}{Y}\,dt.
    \end{align*}
    The SDEs of $X_1$ and $Y$ are given by
    \begin{align*}
        dX_1
        &=\partial_t g_t(x)-d\lambda_1=\left(\frac{2}{X_1}+\frac{2}{X_1-Y}\right)dt-\sqrt{\kappa}\,dB_1+\frac{4}{Y}\,dt,\\
        dY
        &=d\lambda_2-d\lambda_1
        =\sqrt{\kappa}\big(dB_2-dB_1\big)+\frac{8}{Y}\,dt.
    \end{align*}
    The next step is to derive the SDE for $U=X_1/Y$. It\^o's formula and the SDE for $(X_1,Y)$ give the following for $t<\zeta_x$:
    \begin{align*}
        dU&=\frac{1}{Y}\,dX_1-\frac{X_1}{Y^2}\,dY
        +\frac{X_1}{Y^3}\,d[Y,Y]_t-\frac{1}{Y^2}\,d [X_1,Y]_t\\
        &=\frac{1}{Y}\,dX_1-\frac{U}{Y}\,dY
        +\frac{U}{Y^2}\big(2\kappa\,dt\big)-\frac{1}{Y^2}\big(\kappa\,dt\big)\\
        &=\frac{1}{Y}\Bigg[\left(\frac{2}{X_1}+\frac{2}{X_1-Y}\right)dt-\sqrt{\kappa}\,dB_1+\frac{4}{Y}dt\Bigg]
        -\frac{U}{Y}\Bigg[\sqrt{\kappa}(dB_2-dB_1)+\frac{8}{Y}dt\Bigg]+\frac{2\kappa U-\kappa}{Y^2}\,dt.
    \end{align*}
    Now, insert $X_1=UY$ and $X_1-Y=(U-1)Y$ to simplify the first drift term:
    \begin{align*}
        \frac{1}{Y}\left(\frac{2}{X_1}+\frac{2}{X_1-Y}\right)
        =\frac{2}{Y^2}\left(\frac{1}{U}+\frac{1}{U-1}\right).
    \end{align*}
    Collecting the drift terms gives us
    \begin{align*}
        dU
        &=-\frac{\sqrt{\kappa}}{Y}\Big[dB_1+U(dB_2-dB_1)\Big]
        +\frac{1}{Y^2}\left[
        2\left(\frac{1}{U}+\frac{1}{U-1}\right)+(\kappa-4)(2U-1)
        \right]dt.
    \end{align*}
    Next, using $U^{-1}+(U-1)^{-1}=(2U-1){U^{-1}(U-1)^{-1}}$, we obtain
    \begin{align*}
        dU
        &=-\frac{\sqrt{\kappa}}{Y}\Big[(1-U)dB_1+U\,dB_2\Big]
        +\frac{2U-1}{Y^2}\left[
        \frac{2}{U(U-1)}+(\kappa-4)
        \right]dt\\
        &=\frac{\sqrt{\kappa}}{Y}\sqrt{U^2+(1-U)^2}\,dW
        +\frac{2U-1}{Y^2}\left[\frac{2}{U(U-1)}+(\kappa-4)\right]dt
    \end{align*}
    where
    \[dW:=-\frac{(1-U)\,dB_1+U\,dB_2}{\sqrt{U^2+(1-U)^2}}.\]
\end{proof}
\begin{corollary}\label{co:SDEU}
    For $\kappa> 0$, define the increasing time-change 
\begin{align*}
\theta(t):=\int_0^t \frac{\kappa}{ Y^2(u)}du,
\qquad t<\tau_{\text{coll}}=\inf\{u:\ Y_u=0\},
\end{align*}
and write $U(\theta):=U\big(\theta(t)\big)$. The process $U(\theta)$ solves the following SDE on $(0,1)$
\begin{align}
dU(\theta)
&=
\sqrt{U^2(\theta)+(1-U(\theta))^2}\,d W(\theta)
+\frac{2U(\theta)-1}{\kappa}\left[
\frac{2}{U(\theta)(U(\theta)-1)}+(\kappa-4)
\right]d\theta,
\label{eq:Us-timechanged}
\end{align}
for $\theta<\theta(\tau_x)$ and the Brownian motion $W(\theta):=W(\theta(t))$ defined in Lemma~\ref{lm:SDE:U}.
\end{corollary}
\begin{proof}
    The integrand in the time change is positive, so we only need to prove that the time change can go up to infinity, that is, $\theta(t)\to\infty$ as $t\to\tau_{\rm coll}$. $Y(t)=\lambda_2(t)-\lambda_1(t)$ satisfies
\[
dY(t)=\sqrt{2\kappa}\,d\beta(t)+\frac{8}{Y(t)}\,dt
\]
with $\beta(t)$ a standard Brownian motion by Lévy's characterization theorem. Hence, $R(t)=Y(t)/\sqrt{2\kappa}$ is a Bessel process of dimension
\[
\delta=1+\frac{8}{\kappa}.
\]
Moreover,
\[
\theta(t)=\int_0^t\frac{\kappa}{Y^2(u)}\,du
=\frac12\int_0^t\frac{du}{R^2(u)}.
\]
Using Itô's formula,
\[
d\log R(t)=\frac1{R(t)}\,d\beta(t)+\frac{\delta-2}{2R^2(t)}\,dt.
\]
If $\kappa>8$, then $\delta<2$ and $R(t)\to0$ as $t\uparrow\tau_{\text{coll}}$; the logarithmic identity implies $\int_0^{\tau_{\text{coll}}}R_u^{-2}\,du=\infty$, hence $\theta(t)\to\infty$ as $t\uparrow\tau_{\text{coll}}$. If $0<\kappa<8$, or equivalently $\delta>2$, then $\tau_{\text{coll}}=\infty$, and the same Bessel logarithm argument gives $\theta(t)\to\infty$ as $t\to\infty$. Finally, if $\kappa=8$, that is $\delta=2$, the drift term in the SDE of $\log R(t)$ vanishes, making it a local martingale. If
\[
\int_0^\infty \frac{du}{R^2(u)}<\infty,
\]
then the local martingale \(\log R_t\) would have finite total quadratic variation and would converge almost surely. This is incompatible with the long-time behavior of the two-dimensional Bessel process. Hence,
\[
\int_0^\infty\frac{du}{R^2(u)}=\infty
\]
almost surely. Therefore, the time-changed diffusion has the full time interval $[0,\infty)$ available before driver collision. 
\end{proof}
Now, we determine the hitting of $\lambda_1$ or $\lambda_2$ via boundary classification of $U$.  Since $\theta(\cdot)$ is strictly increasing, $U(t)$ hits $\{0,1\}$ in $t$-time if and only if $U(\theta)$ hits $\{0,1\}$ in $\theta$-time. For this purpose, we take a closer look at the scale density and scale function given by
\begin{align*}
\rho(u):=\exp\!\left(-\int\limits_{1/2}^u \frac{2\mu(y)}{\sigma^2(y)}\,dy\right),
\qquad
S(u):=\int\limits_{1/2}^u \rho(v)\,dv, \quad u\in(0,1). 
\end{align*}
and 
\begin{align*}
\sigma^2(u):=u^2+(1-u)^2,
\qquad
\mu(u):=\frac{2u-1}{\kappa}\left[\frac{2}{u(u-1)}+(\kappa-4)\right]
\end{align*}
so that \eqref{eq:Us-timechanged} is $dU(\theta)=\mu(U(\theta))\,d\theta+\sigma(U(\theta))\,d W(\theta)$. 

Before we continue with the next statement, we note that in the case we are considering here, $\tau_x\leq\tau_{\text{coll}}$ almost surely, as the driver trajectories are continuous, and if the drivers collide ($\tau_{\text{coll}}<\infty$, $\kappa>8$), then one of them must hit $g_t(x)$ at time $\tau_{\text{coll}}$ at the latest.
If the drivers do not collide, we have $\tau_{\text{coll}}=\infty$. We will see that, in fact, $\tau_x<\tau_{\text{coll}}$.
\begin{theorem}[Hitting probabilities for the Loewner flow]\label{th:hittingprobs}
Let
\[
u_0:=U(0)=\frac{x-\lambda_1(0)}{\lambda_2(0)-\lambda_1(0)}\in(0,1).
\]
\begin{enumerate}
\item[(i)] If $\kappa>4$, then $\tau_x<\tau_{\text{coll}}$ almost surely with $\tau_\text{coll}=\infty$ for $4<\kappa\leq 8$, $\tau_\text{coll}<\infty$ for $\kappa>8$, and
\begin{align}
\mathbb P\big(g_{\tau_x}(x)=\lambda_1(\tau_x)\big)
=
\frac{\int_{u_0}^1 
\big(v(1-v)\big)^{-\frac{4}{\kappa}}
\big(v^2+(1-v)^2\big)^{\frac{8-\kappa}{\kappa}}
\,dv}{
\int_{0}^1 
\big(v(1-v)\big)^{-\frac{4}{\kappa}}
\big(v^2+(1-v)^2\big)^{\frac{8-\kappa}{\kappa}}
\,dv }.
\label{eq:hit-lambda1-prob-thm}
\end{align}
In particular,
\[
\mathbb P\big(g_{\tau_x}(x)=\lambda_2(\tau_x)\big)
=
1-
\mathbb P\big(g_{\tau_x}(x)=\lambda_1(\tau_x)\big).
\]
\item[(ii)] If $0< \kappa\le 4$, then $\tau_\text{coll}=\infty$, and  $\mathbb P(\tau_x<\infty)=0$.
\end{enumerate}
\end{theorem}
\begin{proof}
    Our main object is the SDE given in Corollary~\ref{co:SDEU}. Note that the time change $\theta(t)$ defined there is shown to be a bijection from $t$-time $[0,\tau_\text{coll})$ to $\theta$-time $[0,\infty)$, with $\tau_\text{coll}<\infty$ if $0<\kappa\leq8$ and $\tau_\text{coll}=\infty$ if $\kappa>8$, so we can work with the scale density $\rho$ and scale function $S$ as given, and any property of $U(\theta)$ is inherited by $U(t)$. The scale density can be expressed as follows,
    \begin{align*}
        \ln \rho(u)&=-\frac2\kappa\int_{1/2}^u\frac{2y-1}{y^2+(1-y)^2}\left[\frac{2}{y(y-1)}+(\kappa-4)\right]\, dy\\
        &=-\frac2\kappa\int_{0}^{2u-1}\frac{v}{1+v^2}\left[(\kappa-4)-\frac{8}{1-v^2}\right]\, dv.
    \end{align*}
    The second line is obtained by performing the substitution $v:=2y-1$. The first term is readily integrated,
    \begin{align*}
        -\frac{2(\kappa-4)}\kappa\int_{0}^{2u-1}\frac{v}{1+v^2}\, dv=-\frac{\kappa-4}{\kappa}\ln[2(2u^2-2u+1)],
    \end{align*}
    and the second term reads
    \begin{align*}
        \frac{16}\kappa\int_{0}^{2u-1}\frac{v}{1-v^4}\, dv&=\frac{8}\kappa\int_{0}^{2u-1}\left[\frac{v}{1+v^2}-\frac{1}{2(1+v)}+\frac{1}{2(1-v)}\right]\, dv\\
        &=\frac4\kappa\ln(4u^2-4u+2)-\frac4\kappa\ln(2u)-\frac4\kappa\ln[2(1-u)]\\
        &=\frac4\kappa\ln\frac{2u^2-2u+1}{2u(1-u)},
    \end{align*}
    so both terms add up to
    \begin{align*}
        \ln \rho(u)&=-\frac{\kappa-4}{\kappa}\ln\big[2(2u^2-2u+1)\big]+\frac4\kappa\ln\frac{2u^2-2u+1}{2u(1-u)}\\
        &=-\ln2+\frac{8-\kappa}{\kappa}\ln(u^2+(1-u)^2)-\frac4\kappa\ln[u(1-u)],
    \end{align*}
    or
    \begin{align*}
        \rho(u)=\frac12(u^2+(1-u)^2)^{(8-\kappa)/\kappa}\big[u(1-u)\big]^{-4/\kappa}.
    \end{align*}
    By inspection of this function, we see that the scale function, which is the integral of $\rho$, diverges at $u=0$ or $1$ when $\kappa\leq4$ because of the term $[u(1-u)]^{-4/\kappa}$. This, by \cite[Sec. 11]{Feller} and Corollary~\ref{co:SDEU}, yields \emph{(ii)}.
    
    When $\kappa>4$, the scale function does not diverge, meaning that $\theta(\tau_x)$ is finite, namely $\theta(\tau_x)<\theta(\tau_\text{coll})=\infty$, so here it becomes clear that $\tau_x<\tau_\text{coll}$, and the probability of hitting $\lambda_1$ is given by \eqref{eq:hit-lambda1-prob-thm} due to the optional stopping theorem. This happens at a finite $t$-time $\tau_x$ by Corollary~\ref{co:SDEU}, where it is shown that $\theta$ is a bijection between $t\in[0,\tau_\text{coll})$ and $\theta(t)\in[0,\infty)$, with $\tau_\text{coll}<\infty$ when $\kappa>8$ and $\tau_\text{coll}=\infty$ when $0<\kappa\leq8$. 
\end{proof}

From this result, we can observe two features.
The first is that we recover the phases outlined in Theorem~\ref{thm:classical-interface-transition}, but the conditions for this statement are stronger in the sense that whenever there is a collision in finite time, we know that there is a hitting event that precedes it.
This is clear from the setting of our statement, and the additional strength of the assumption allows us to write the result for any $\kappa$ positive, not only $0<\kappa\leq8$.
Indeed, when there are more than two drivers, there may be collisions between drivers that do not surround $g_t(x)$, in which case one may have $\tau_\text{coll}<\tau_x$.

The second feature is that the scale density $\rho(u)$ turns out to be symmetric wrt $u=1/2$.
This is because the $A_1$ ($N=2$ Dyson) drivers enjoy left-right symmetry in law, so we can expect that the probability of hitting by $\lambda_1$ after starting from $u_0$ should be the same as the probability of hitting by $\lambda_2$ after starting from $1-u_0$.
We do not expect this second feature to be present in the $B_2$ or $D_2$ cases when $\lambda_1(0)<x<\lambda_2(0)$, as in those cases there exists an additional drift from $-\lambda_1$ and $-\lambda_2$ (and also from the origin in the $B_N$ case) which drives the process toward $\lambda_2$, so in that case the scale density should be non-symmetric.

\section{Application of the method for two drivers in a  deterministic setting}

So far our Symmetric Polynomials method was applied to models with random drivers which have applications in Physics, and beyond. However, there is a large part of the literature devoted to the deterministic study of the Loewner and Multiple Loewner equations \cite{Kager04}. In addition to that there are works considering the noise parameter going to zero and investigating the models in those regimes; this approach also appeared in {the literature of particle systems under the name of freezing limits, see for instance \cite{AHV,AKM12,AKM14,AV19B,AV19A,HV21,V19,Voit2023,VW}, for a survey. It is well-known (see for instance \cite{AKM12,AKM14,Deift}) that when $\kappa\to0$ the evolution of a Bessel-Dunkl particle system configuration ($b_i\equiv0$, $k_\alpha\equiv c^2>0$) satisfies an ODE which asymptotically tends to $C\sqrt{t}v$, where $C>0$ is a constant and $v$ is a vector comprised of the zeros of the $N$th degree polynomial of a classical orthogonal family, Hermite zeros for $R=A_{N-1}$ and square root of Laguerre zeros for $R=B_N$. In this section, we consider the Dyson case, namely $R=A_{N-1}$, where if $N=2$, the zeros of the 2nd Hermite polynomial lie at $\pm1/\sqrt2$, so we can write the asymptotic form of the drivers' trajectories as $\pm C\sqrt t$, where $C$ absorbs the factor $\sqrt2$.}

As we have seen so far, the method of Symmetric Polynomials is very helpful when one wants to study or deduce phase transitions of the models in terms of the hitting times of the drivers. In the deterministic case the phase transition is understood  for some choice of drivers (see \cite{Kager04}). To the best of our knowledge less is known about these phase transitions for the multiple Loewner dynamics. We explore some applications of this method for some choices of drivers in the deterministic setting along with some model when the noise parameter goes to zero, that is, the noise is tuned-down in the random system.  

Given the similar structure with the models we have studied before in the random case, we first investigate the method for the model in the following section. In order to follow the same theme, we perform the analysis for $N=2$. One can consider general $N$ models as well. 
\subsection{Multiple Loewner evolution driven by two deterministic repulsive particles}

Here, we take $\kappa=0$ in Lemma~\ref{lm:SDE:U}, but for the sake of added generality we modify the driver evolution given in \eqref{eq:general-root-adapted-loewner} by setting $k_\alpha= c^2$, namely,
\[\frac{d\lambda_1}{dt}=\frac{c^2}{\lambda_1-\lambda_2},\quad \frac{d\lambda_2}{dt}=\frac{c^2}{\lambda_2-\lambda_1},\quad c^2>0.\]
With this setup, we claim the following.
\begin{lemma}
    \label{lm:deterministic:U}
    Consider $U(t)$ as defined in Lemma~\ref{lm:SDE:U} with $\kappa=0$ and the drivers $\lambda_1$, $\lambda_2$ as given above and started from $-C$ and $+C$ respectively ($C>0$) without loss of generality, together with the time transformation
    \[s(t):=\frac1{4c^2}\ln\Big(\frac{C^2+c^2t}{C^2}\Big).\]
    Then, $U(s)$ satisfies the ODE
    \[\frac{dU}{ds}=(2U-1)\Big(\frac2{U(U-1)}-c^2\Big).\]
\end{lemma}
\begin{proof}
We note that
\[\frac{d}{dt}(\lambda_1+\lambda_2)=\frac{c^2}{\lambda_1-\lambda_2}+\frac{c^2}{\lambda_2-\lambda_1}=0,\]
so $\lambda_1+\lambda_2$ is constant, and equal to zero because $\lambda_1(0)+\lambda_2(0)=-C+C=0$, so for every $t$, $\lambda_2=-\lambda_1$. We can readily solve the drivers' differential equations to obtain
\[\lambda_2(t)=\sqrt{C^2+c^2t},\quad \lambda_1(t)=-\sqrt{C^2+c^2t}.\]
Following a modification of Lemma~\ref{lm:SDE:U} ($k_\alpha\equiv c^2$), we see that 
\[\frac{dU}{dt}=\frac{2U-1}{4(C^2+c^2t)}\Big(\frac2{U(U-1)}-c^2\Big)\]
(compare the denominator in the coefficient with $Y^2=(\lambda_2-\lambda_1)^2$ in Lemma~\ref{lm:SDE:U}). Using the time transformation we see that the time derivative of $U$ becomes
\[\frac{dU}{dt}=\frac{ds}{dt}\frac{d}{ds}U(s(t))=\frac{d}{dt}\Big[\frac1{4c^2}\ln\Big(\frac{C^2+c^2t}{C^2}\Big)\Big]\frac{d}{ds}U(s(t))=\frac1{4(C^2+c^2t)}\frac{d}{ds}U(s(t)),\]
so substituting this in the previous expression we obtain the ODE in $s$-time given in the claim.
\end{proof}
Note that in all regimes, the ODE in Lemma~\ref{lm:deterministic:U} has well-defined asymptotics and does not show collisions. In the case where $U(0)>1$, the ODE drives $U$ toward $(1+\sqrt{1+8/c^2})/2$. If $0<U(0)<1$, then $U$ approaches $1/2$, whereas for $U(0)<0$, $U$ tends toward $(1-\sqrt{1+8/c^2})/2$. We also note that setting $c^2\downarrow0$ fixes the drivers at $\pm C$, so in that case the time transformation is unnecessary and the only stable point is $U\to1/2$, which occurs when $0<U(0)<1$.

\subsection{The Loewner Equation with Two Symmetric Driving Points}

Let us consider the case where $C=0$. We consider the multiple Loewner ODE 
driven by two symmetric drivers $\lambda_1(t) = +c\sqrt{t}$ and
$\lambda_2(t) = -c\sqrt{t}$, $c > 0$:
\begin{equation}\label{eq:loewner1}
  \partial_tg_t(x)
  = \frac{2}{g_t(x) - \lambda_1(t)} + \frac{2}{g_t(x) - \lambda_2(t)}
  = \frac{2}{g_t(x) - c\sqrt{t}} + \frac{2}{g_t(x) + c\sqrt{t}}
  = \frac{4g_t(x)}{g_t(x)^2 - c^2 t},
\end{equation}
where the last equality uses the partial-fraction identity and holds for all
$t$ such that $g_t(x) \notin \{\lambda_1(t), \lambda_2(t)\}$.
This means that we must have $x\neq 0$ in order to have a non-trivial solution for $g_t(x)$, and therefore we have either $x>\lambda_1(0)=\lambda_2(0)$ or $x<\lambda_1(0)=\lambda_2(0)$; in both cases the product $(x-\lambda_1(0))(x-\lambda_2(0))$ is positive, and it remains positive until a driver hits $g_t(x)$.

For a particle $g_t(x)$ satisfying~\eqref{eq:loewner1}, define the two
symmetric polynomials in the translated variables $(g_t(x) - \lambda_1)$ and
$(g_t(x) - \lambda_2)$: 

\begin{align*}
  P(t) &:=\sqrt{E_2}=(g_t(x) - \lambda_1(t))(g_t(x) - \lambda_2(t)) = g_t^2(x) - c^2 t, \\%\label{eq:P}\\[4pt]
  S(t) &:= E_1 =  (g_t(x) - \lambda_1(t))^2 + (g_t(x) - \lambda_2(t))^2 = 2g_t^2(x) + 2c^2 t. %\label{eq:S}
\end{align*}

$P(t)$ is properly defined due to the requirement that $x\neq0$, as mentioned above. These two quantities are related by an exact algebraic identity:
\begin{equation}\label{eq:identity1}
  S(t) = 2P(t) + 4c^2 t, \qquad t \geq 0.
\end{equation}
Hence, $S$ and $P$ are \emph{not independent}: $S$ is an affine function of
$P$ with a strictly increasing time-dependent shift $4c^2 t$.

\begin{proposition}\label{prop:ODEs}
Along any solution of~\eqref{eq:loewner1}, the symmetric polynomials
$P(t)$ and $S(t)$ satisfy the following first-order ODEs
\begin{align}
  \frac{dP}{d t} &= \frac{(8 - c^2)\,g_t^2(x) + c^4 t}{P(t)}, \label{eq:Pdot1}\\[6pt]
  \frac{dS}{d t} &= 2\,\frac{dP}{d t} + 4c^2. \label{eq:Sdot1}
\end{align}
\end{proposition}

\begin{proof}
Differentiating $P = g_t^2(x) - c^2 t$ and substituting~\eqref{eq:loewner1}:
\begin{align*}
  \frac{dP}{d t} =& 2g_t(x)\,\partial g_t(x) - c^2
            = 2g_t(x) \cdot \frac{4g_t(x)}{g_t^2(x) - c^2 t} - c^2
            = \frac{8g_t^2(x)}{P} - c^2\\
            =& \frac{8g_t^2(x) - c^2 P}{P}
            = \frac{8g_t^2(x) - c^2(g_t^2(x) - c^2 t)}{P}
            = \frac{(8 - c^2)g_t^2(x) + c^4 t}{P}.
\end{align*}
 Equation~\eqref{eq:Sdot1} 
follows immediately from differentiating the
identity~\eqref{eq:identity1}. 
\end{proof}

% \noindent Writing $u(t) = P(t)^2$, equation~\eqref{eq:Pdot1} becomes
% \begin{equation}\label{eq:u}
%   \frac{du}{dt} = 2\bigl[(8 - c^2)\,g_t^2(x) + c^4 t\bigr],
% \end{equation}
% a Bernoulli-type relation linking the growth rate of $P^2$ to the
% competition between the Loewner repulsion (coefficient 8) and the
% driving spread (coefficient $c^2$).

Recall that $\tau_x$ denotes the first hitting time of $g_t(x)$ by any of the drivers $\lambda_1$ or $\lambda_2$ and that the first time that drivers collide is $\tau_{\text{coll}}=\infty$. 
%$Then, the sign of ${dP}/dt$ is entirely controlled by the numerator
% $(8 - c^2)g_t^2(x) + c^4 t$ in~\eqref{eq:Pdot1}.
The ODE of $P(t)$ can be rewritten as follows,
\begin{align*}
    \frac{dP}{dt}=8-c^2+\frac{8c^2t}{P(t)},
\end{align*}
which is a closed ODE. By inspection, one notices that \(\tau_x=\infty\), since for any initial condition \(P(0)>0\), the rational term ensures that the right-hand side becomes positive as \(P(t)\) approaches zero. Hence, there are no collisions.

We can also consider other models, for instance a multiple Loewner ODE on $[0,1)$ driven by two symmetric
points that shrink to zero:
\begin{equation}\label{eq:drivers}
  \lambda_1(t) = +c\sqrt{1-t}, \qquad \lambda_2(t) = -c\sqrt{1-t}, \qquad c > 0.
\end{equation}
The equation of motion is
\begin{equation}\label{eq:loewner}
  \partial_tg_t(x)
  = \frac{2}{g_t(x) - c\sqrt{1-t}} + \frac{2}{g_t(x) + c\sqrt{1-t}}
  = \frac{4g_t(x)}{g_t^2(x) - c^2(1-t)}.
\end{equation}
Using similar definitions for $P$ and $S$, we arrive at the following ODEs,
\begin{align}
  \frac{dP}{dt} &= \frac{(8+c^2)\,g_t^2(x) - c^4(1-t)}{P(t)},
  \label{eq:Pdot}\\[6pt]
  \frac{dS}{dt} &= 2\,\frac{dP}{dt} - 4c^2.
  \label{eq:Sdot}
\end{align}

Performing similar calculations, we obtain the following ODE for $P$,

\begin{align}
  \frac{dP}{dt}&= (8+c^2) + \frac{8c^2(1-t)}{P(t)}.
  \label{eq:Pself}
\end{align}

Both terms on the right hand-side of~\eqref{eq:Pself} are strictly positive for
$t\in[0,1)$ and $P>0$.  Hence $d{P}/dt > 0$ for all such $t$,
and $P$ is strictly increasing. Thus, we see that, in this case as well,  $g_t(x)$ never reaches either driver. 

Thus, we have studied deterministic versions of the symmetric polynomial method in quite simple examples, but these techniques are readily available and generalized to other deterministic models.

\section{Future directions}\label{sec:outlook}

One potential extension of the symmetric polynomial methods applied to multiple SLE is to consider the reverse-time (backward) multiple Loewner equation, see \cite{Vladweld}, for a one-particle version of the corresponding dynamics. 
In the reverse-time case, one is interested in studying the hitting times of the drivers (in the corresponding regime of parameters for which they are finite a.s.), see \cite{Vladweld}. The same question is interesting in the multiple Loewner differential equation case as well. For the one-curve case,  the problem of continuity of the hitting times (via the welding homeomorphism) is linked with the famous still open problem of the continuity in $\kappa$ problem for the SLE curves and related objects, see \cite{Vladcont,FrizYuan,Vikcont} for recent progress in this direction.

The question of continuity in the parameter $\beta\geq1$ is interesting also for the forward Multiple SLE, see \cite{JiamVlad} for the analysis of continuity in $\beta\geq 1$ of these maps for the case $N=2$ in the Carathe\'{o}dory sense.

In addition, it will be interesting to study the interactions between the technique of symmetric polynomials that we develop for Multiple SLE in this manuscript and the corresponding resolvent analysis developed in \cite{KaVlad} which uses random matrices. One can try to obtain a similar resolvent analysis for the extended maps on the real line and then relate the behavior of the symmetric polynomials with the singular values of the matrix $A_t-g_t(x)I_n$, where $A_t$ is a time-evolving $N \times N$ (random) matrix and $I_n$ is the $N\times N$ matrix, following \cite{KaVlad}. In addition, one can consider other observables and more general matrix dynamics (such as adding a drift, etc.) and their corresponding resolvents and study the corresponding symmetric polynomials for those models. In addition, given the current developments in the area of neural networks and SLE, see \cite{shrotri2026neural}, it is interesting to consider an extensions of these methods to multiple Loewner models, and to the learning of parameters of the corresponding symmetric polynomials. 

A natural extension of the analysis in our paper is to consider the case of $N\geq 3$ Dyson particles as drivers. Our initial investigations on this direction reveal that the dynamics is more complicated and a similar analysis based on a  diffusion on the interval $[0, 1]$ is not direct. It is known that the case $N=2$ caries special structure \cite{JiamVlad}. However, we believe that this method based on symmetric polynomials will be effective in those cases as well, with more-involved and detailed analysis of the corresponding process.

A natural follow-up of this method will consist also in the analysis of the complex valued dynamics $g_t(z).$

\section*{Acknowledgments}
NH was supported by the Probability and Statistics Group of the German Mathematical Society for a trip to the USA to work on this paper. 

\printbibliography
\end{document}